\documentclass[12pt,reqno]{amsart}

\usepackage[T1]{fontenc}
\usepackage{amsmath,amssymb,amsfonts,amsthm}
\usepackage[margin=1.5in]{geometry}
\usepackage{etoolbox}
\usepackage{hyperref}
\hypersetup{
  hidelinks,
  pdftitle={Average root numbers in two isotrivial families of elliptic curves},
  pdfauthor={Yijie Diao},
  pdfsubject={Average root numbers in two isotrivial elliptic families},
  pdfkeywords={average root numbers; isotrivial elliptic families; binary forms; polynomial values; quantitative transference principle; Euler products; del Pezzo surfaces of degree 1}
}

\AtEndEnvironment{thebibliography}{\enlargethispage{4\baselineskip}}

\numberwithin{equation}{section}

\newtheorem{theorem}{Theorem}[section]
\newtheorem{lemma}[theorem]{Lemma}
\newtheorem{proposition}[theorem]{Proposition}
\newtheorem{corollary}[theorem]{Corollary}
\theoremstyle{definition}
\newtheorem{definition}[theorem]{Definition}
\newtheorem{remark}[theorem]{Remark}

\renewcommand{\leq}{\leqslant}
\renewcommand{\le}{\leqslant}
\renewcommand{\geq}{\geqslant}
\renewcommand{\ge}{\geqslant}

\newcommand{\sumstar}{\sideset{}{^*}\sum}

\newcommand{\PP}{\mathbb{P}}
\newcommand{\FF}{\mathbb{F}}
\newcommand{\ZZ}{\mathbb{Z}}
\newcommand{\NN}{\mathbb{N}}
\newcommand{\QQ}{\mathbb{Q}}
\newcommand{\RR}{\mathbb{R}}
\newcommand{\CC}{\mathbb{C}}

\newcommand{\cc}{\mathbf{c}}

\DeclareMathOperator{\rank}{rank}
\DeclareMathOperator{\Gal}{Gal}
\DeclareMathOperator{\sign}{sign}
\DeclareMathOperator{\cont}{cont}

\newcommand{\rw}{W}

\renewcommand{\pmod}[1]{\,(\mathrm{mod}\,#1)}

\begin{document}
\raggedbottom

\title[Average root numbers in isotrivial families]
{Average root numbers in two isotrivial families of elliptic curves}
\author{Yijie Diao}
\address{ISTA\\
Am Campus 1\\
3400 Klosterneuburg\\
Austria}
\email{yijie.diao@ist.ac.at}

\begin{abstract}
We study root numbers in the isotrivial families $y^2=x^3+a$ and $y^2=x^3+ax$.
For a broad class of fixed binary forms, we prove that the average root number over primitive pairs exists.
Assuming finiteness of the relevant Tate--Shafarevich groups, we deduce Zariski density for certain del Pezzo surfaces of degree $1$ arising from separable binary sextics.
We also establish explicit averages of root numbers for almost all polynomials of each fixed degree $d\geq2$, ordered by coefficient height.
The proof develops a new quantitative transference principle for polynomial values.
\end{abstract}

\maketitle

\setcounter{tocdepth}{1}
\tableofcontents

\section{Introduction}

A del Pezzo surface is a smooth projective surface with ample anticanonical class.
Every del Pezzo surface of degree $1$ over $\QQ$ has a rational point, namely the unique base point of its anticanonical pencil \cite[Section~2A]{VA}.
It is nevertheless unknown whether the rational points on every such surface are Zariski dense.
Blowing up the base point produces an elliptic surface over $\PP^1$, so the density problem can be approached through the arithmetic of its fibres.
The root number of an elliptic curve is the sign in the functional equation of its $L$-function.
The parity conjecture predicts that it is $(-1)^{\rank E(\QQ)}$.
Thus variation of root numbers in the fibres can provide evidence for rank variation and, in favourable cases, for the density of rational points on the surface.
To place our results in context, we first recall what is known about rational points on del Pezzo surfaces of degree $1$ and about root numbers in algebraic families.

For a del Pezzo surface of degree at least $3$, a rational point over an infinite field implies unirationality, by work of Segre \cite{Seg1,Seg2}, Manin \cite{Man}, Koll\'ar \cite{Kol}, and Pieropan \cite{Pie}.
In degree $2$, Salgado--Testa--V\'arilly-Alvarado \cite{STVA} gave analogous results under conditions on the point.
Degree $1$ is substantially less understood.
Koll\'ar and Mella \cite[Theorem~7]{KM} proved unirationality in the degree $1$ conic bundle case, but even Zariski density remains open in general.
The conjectural picture first formulated by Colliot-Th\'el\`ene and Sansuc \cite{CTS} and further discussed by Colliot-Th\'el\`ene \cite{CTAWS,CT-pencil} predicts that the Brauer--Manin obstruction is the only obstruction to the Hasse principle and weak approximation for smooth proper geometrically rational varieties; in particular, it predicts Zariski density whenever the Brauer--Manin set is nonempty.

Known density results for degree $1$ surfaces use several complementary methods.
Explicit parametrisations and base changes are used in work of Ulas \cite{Ula1,Ula2} and Jabara \cite{Jab}.
Using explicit constructions of multisections, Salgado--van~Luijk \cite[Theorem~1.4]{SvL} give several criteria for Zariski density, including the existence of a nodal fibre over a rational point of the base.
Desjardins--Winter \cite{DWi} prove density for a specific family using low-genus curves.
In his master's thesis, Nijgh \cite{Nijgh} extends their construction to a larger family, for which Demeio--Streeter--Winter \cite{DeStWi} establish the Hilbert property under corresponding hypotheses.
Desjardins--Jovanovic \cite{DJ25} construct further trisections of low genus.
Rank jumps on rational elliptic surfaces were studied by Salgado \cite{Sal12}, while Loughran--Salgado \cite{LS22} proved non-thinness results for rank-jumping fibres.
Other results include an effective study of diagonal sextics by Desjardins--Naskr\k{e}cki \cite{DN24}.
For rational isotrivial elliptic surfaces with nonzero $j$-invariant, Desjardins \cite{Des} proved Zariski density by geometric methods; this applies to the elliptic surfaces obtained from the singular $j=1728$ models considered below.

The systematic study of local root numbers in algebraic families was developed by Rohrlich \cite{Roh93,Roh96}.
For isotrivial families with potentially good reduction, Liverance \cite{Liv95} gave an early closed formula in the $j=0$ case, Halberstadt \cite{Hal98} treated the difficult primes $2$ and $3$, and Rizzo \cite{Riz03} developed explicit local and average formulas.
Manduchi \cite{Man95} and Grant--Manduchi \cite{GM97,GM98} connected such calculations to elliptic surfaces.
V\'arilly-Alvarado \cite{VA08,VA} first studied weak approximation and later introduced the root number method for families with $j$-invariants $0$ and $1728$.
Related topological questions over the reals for fibres of positive rank were studied by Kuwata--Wang \cite{KW93} and Munshi \cite{Mun07,Mun10}.

Helfgott's general framework \cite{Hel} expresses a family root number as a product of local arithmetic functions.
Subject to cancellation of Chowla type and squarefree value estimates for the polynomials of bad reduction, it gives zero average for certain families.
Desjardins \cite{DesVar,DesTwist,DesInt} obtained related variation results for rational and integer parameters and gave explicit criteria for polynomial twist families.
When this multiplicative-reduction mechanism is absent, however, the average need not vanish.
Bettin--David--Delaunay \cite{BDD18} classified six essentially distinct classes of non-isotrivial, potentially parity-biased families whose Weierstrass coefficients $a_2,a_4,a_6$ have degree at most $2$, and computed averages for two representative subfamilies of one class; Chinis \cite{Chi} subsequently computed the average for that entire class.
At the extremal end of this phenomenon, Chu--Desjardins \cite{ChuDes} gave necessary and sufficient arithmetic conditions for certain linear subfamilies to have constant root number on every integer fibre, extending Washington's classical example, and obtained partial analogues for a related quadratic class.

An earlier average of Rizzo \cite{Riz99}, ordered by the height of a rational twist parameter, naturally leads to primitive pairs after the parameter is written as a coprime numerator and denominator.
The squarefree-sieve approach to elliptic twists was initiated by Gouv\^ea--Mazur \cite{GM91}.
For binary forms whose irreducible factors have degree at most $6$, Greaves \cite{Gre92} established a squarefree sieve.

In this paper, we study the two isotrivial families
\begin{equation}\label{eq:intro-families}
y^2=x^3+F(z,w),
\quad
y^2=x^3+G(z,w)x,
\end{equation}
whose generic $j$-invariants are $0$ and $1728$, respectively.
For $a\in\ZZ\setminus\{0\}$, let $\rw(a)$ and $\rw^\star(a)$ denote the root numbers of
\[
y^2=x^3+a
\quad\text{and}\quad
y^2=x^3+ax,
\]
respectively; put $\rw(0)=\rw^\star(0)=0$.
For $x\ge1$, set
\begin{equation}\label{eq:def-Sprim}
S_{\mathrm{prim}}(x)
=\{(m,n)\in\ZZ^2:1\le m,n\le x,\ \gcd(m,n)=1\}.
\end{equation}
For an irreducible homogeneous form $h\in\QQ[z,w]$, let
\begin{equation}\label{eq:def-Kh}
K_h=
\begin{cases}
\QQ[T]/(h(T,1)),&\text{if $h(T,1)$ is nonconstant},\\
\QQ,&\text{if $h(T,1)$ is constant},
\end{cases}
\end{equation}
the residue field of the corresponding closed point of $\PP^1_\QQ$.
When $H\in\ZZ[z,w]$ is a nonzero homogeneous form, we write
\begin{equation}\label{eq:intro-factorisation}
H=c\prod_{i=1}^r h_i^{e_i},
\end{equation}
where $c\in\ZZ\setminus\{0\}$, the exponents $e_i$ are positive integers, and the $h_i$ are pairwise nonassociate primitive irreducible homogeneous forms.

We first consider fixed binary forms and average their root numbers over primitive pairs.
\begin{theorem}\label{thm:asy}
Let $F,G\in\ZZ[z,w]$ be nonzero homogeneous binary forms.
\begin{enumerate}
\item[(i)] Write $F=c\prod_{i=1}^r f_i^{e_i}$ as in \eqref{eq:intro-factorisation}.
Suppose that $\deg f_i\le6$ whenever $3\nmid e_i$.
Then there is a constant $C_F$ with $|C_F|\le1$ such that
\[
\lim_{x\to\infty}
\frac{\#\{(m,n)\in S_{\mathrm{prim}}(x):\rw(F(m,n))=\pm1\}}
{\#S_{\mathrm{prim}}(x)}
=\frac{1\pm C_F}{2}.
\]
If there exists a factor $f_i^{e_i}$ of $F$ satisfying
\[
3\nmid e_i
\quad\text{and}\quad
\QQ(\sqrt{-3})\not\subset K_{f_i},
\]
then $|C_F|<1$.

\item[(ii)] Write $G=c\prod_{i=1}^r g_i^{e_i}$ as in \eqref{eq:intro-factorisation}.
Suppose that $\deg g_i\le6$ whenever $e_i$ is odd.
Then there is a constant $C_G^\star$ with $|C_G^\star|\le1$ such that
\[
\lim_{x\to\infty}
\frac{\#\{(m,n)\in S_{\mathrm{prim}}(x):\rw^\star(G(m,n))=\pm1\}}
{\#S_{\mathrm{prim}}(x)}
=\frac{1\pm C_G^\star}{2}.
\]
If there exists a factor $g_i^{e_i}$ of $G$ satisfying
\[
e_i\ \text{odd}
\quad\text{and}\quad
\QQ(i)\not\subset K_{g_i},
\]
then $|C_G^\star|<1$.
\end{enumerate}
\end{theorem}

The residue-field condition in each part is a sufficient condition supplied by the method; we do not claim that it is necessary.
It allows the Chebotarev argument to produce a prime at which opposite local signs occur on sets of positive density.
In either part, strict inequality implies that both signs occur with positive density.
The condition on the degrees of the factors is automatic for sextic forms in Theorem~\ref{thm:asy}(i) and quartic forms in Theorem~\ref{thm:asy}(ii).
Apart from this condition, no fixed prime divisor hypothesis is imposed.
Thus the factor-degree condition is precisely the range furnished by Greaves's squarefree sieve, while the conclusion strengthens root-number variation to a full limiting distribution.

For a separable binary sextic $F$, the surface
\begin{equation}\label{eq:DP1-family}
X_F:\quad y^2=x^3+F(z,w)\subset\PP(2,3,1,1)
\end{equation}
is a smooth del Pezzo surface of degree $1$.
Its anticanonical pencil has one rational base point, and blowing up that point gives an elliptic surface with generic $j$-invariant $0$.
V\'arilly-Alvarado \cite[Theorem~2.1]{VA} used a sieve method to prove conditional Zariski density for squarefree sextics under certain hypotheses.
Theorem~\ref{thm:asy} strengthens this result by giving the full average and removing the fixed prime divisor hypothesis.

\begin{corollary}\label{cor:conditional-density-j0}
Let $F\in\ZZ[z,w]$ be a separable binary sextic.
Suppose that some irreducible factor $f$ of $F$ satisfies
\[
\QQ(\sqrt{-3})\not\subset K_f.
\]
If the Tate--Shafarevich groups of elliptic curves over $\QQ$ with $j$-invariant $0$ are finite, then $X_F(\QQ)$ is Zariski dense in $X_F$.
\end{corollary}

Desjardins \cite[Theorem~6.1]{Des} analysed diagonal $j=0$ sextic families and classified their exceptional constant-root-number cases.
Corollary~\ref{cor:conditional-density-j0} overlaps with that work on the diagonal subfamily.
Its scope is different: it applies to arbitrary separable binary sextics satisfying the residue-field hypothesis and obtains conditional density from a positive density of negative-root-number fibres.

\begin{remark}
By the $p$-parity theorem of Dokchitser--Dokchitser \cite{DoDo}, the finiteness hypothesis implies that every smooth fibre with root number $-1$ has positive Mordell--Weil rank.
Although it is not claimed to be necessary, the residue field hypothesis in Corollary~\ref{cor:conditional-density-j0} cannot simply be omitted from a general root-number criterion of this form.
Indeed, V\'arilly-Alvarado \cite[Example~7.1]{VA} considers
\[
F(z,w)=27z^6+16w^6,
\]
for which the field hypothesis fails and every fibre
\[
y^2=x^3+27m^6+16n^6,
\quad [m:n]\in\PP^1(\QQ),
\]
has root number $+1$.
V\'arilly-Alvarado nevertheless proves Zariski density for this particular surface by exhibiting two independent non-torsion sections. 
\end{remark}

\medskip
For a separable binary quartic $G$, the surface
\begin{equation}\label{eq:DP1-family-1728}
X_G:\quad y^2=x^3+G(z,w)x\subset\PP(2,3,1,1)
\end{equation}
is not smooth and hence is not a del Pezzo surface under the convention used in \cite[Section~2B]{VA}.
Over $\overline{\QQ}$, it has four ordinary double points of type $A_1$, one above each zero of $G$.
Indeed, since $G$ is separable, we may use $t=G(z,w)$ as a local parameter at such a zero; the local equation is then $y^2=x^3+tx$, and the change of variable $s=t+x^2$ gives $y^2=xs$, the standard $A_1$ normal form.
Thus, $X_G$ is a singular del Pezzo surface of degree $1$, and its minimal resolution is a generalized del Pezzo surface of degree $1$.
Resolving these points and blowing up the base point gives a non-trivial rational elliptic surface $\mathcal E_G$ with generic $j$-invariant $1728$.
Desjardins \cite[Theorem~5.4]{Des} proved unconditionally that every such surface has a Zariski dense set of rational points by a geometric argument.

Theorem~\ref{thm:asy} keeps the binary form fixed and lets the primitive pair vary.
A complementary question is whether one can obtain uniform information when the polynomial itself also varies.
Ordering polynomials by coefficient height makes it possible to work in every fixed degree after excluding a density-zero set, although the length of the average is then tied to the coefficient height.

This ``almost all'' viewpoint has proved powerful in several related problems.
Skorobogatov--Sofos \cite{SS23} proved Schinzel's hypothesis for $100\%$ of polynomials, while Ter\"av\"ainen \cite{Ter24} proved sign changes for broad classes, including every polynomial outside a set of density zero, together with stronger cancellation in certain factorisable cases.
The squarefree-value estimates needed to control large prime squares have a parallel averaging theory: Filaseta \cite{Fil92} initiated their study over polynomials, and quantitative results were subsequently obtained by Shparlinski \cite{Shp13}, Browning--Shparlinski \cite{BrSh24}, Jelinek \cite{Jel}, and Sofos \cite{Sof26}.
On the oscillatory side, Wilson \cite{Wil25} proved Gaussian moment results for polynomial Chowla sums, while Kravitz--Woo--Xu \cite{KWX25} obtained distributional results for prime values and Liouville sign patterns.
Related averages of arithmetic functions appear in work of Destagnol--Sofos \cite{D-S-2} and Diao \cite{DiaoRandomBinary}.

A key technical tool is the quantitative transference theorem of Browning--Sofos--Ter\"av\"ainen \cite[Theorem~2.2]{BST}, which passes suitable estimates in arithmetic progressions to the values of almost all polynomials.
Theorem~\ref{thm:main} adapts this mechanism to the full root number, whose local factors can produce a nonzero mean.
To state it, fix $d\ge2$ and write
\begin{equation}\label{eq:def-gc}
\begin{gathered}
\cc=(c_0,\ldots,c_d),
\quad
g_\cc(t)=c_d t^d+\cdots+c_1 t+c_0,
\quad
|\cc|=\max_i|c_i|,
\end{gathered}
\end{equation}
and define
\begin{equation}\label{eq:def-SH}
S(H)=\{\cc\in\ZZ^{d+1}:|\cc|\le H,\ c_0\ne0\}.
\end{equation}

\begin{theorem}\label{thm:main}
Let $d\ge2$ and let $H$ be sufficiently large in terms of $d$.
Put
\[
x=H^{1/(4d)},
\quad
\theta=\frac{\log\log x}{\log\log\log x}.
\]
There is a set $\mathcal E_d(H)\subset S(H)$ satisfying
\[
\#\mathcal E_d(H)\ll\theta^{-1/3}H^{d+1}
\]
and containing every vector with $c_d=0$.
For every $\cc\in S(H)\setminus\mathcal E_d(H)$, there are real numbers $\kappa_\cc$ and $\kappa_\cc^\star$ satisfying $|\kappa_\cc|,|\kappa_\cc^\star|<1$ and
\begin{align}
\sum_{1\le n\le x}\rw(g_\cc(n))
&=\kappa_\cc x+O\bigl(\theta^{-1/3}x\bigr),
\label{eq:main-fixed}\\
\sum_{1\le n\le x}\rw^\star(g_\cc(n))
&=\kappa_\cc^\star x+O\bigl(\theta^{-1/3}x\bigr).
\label{eq:main-fixed-star}
\end{align}
\end{theorem}

The constants in Theorem~\ref{thm:main} are given by explicit local factors: $2$- and $3$-adic factors and a convergent Euler product; see \eqref{eq:kappa-c-product} and \eqref{eq:kappa-c-product-star}, respectively.
Since $\#S(H)\asymp H^{d+1}$, the exceptional set has density zero in the coefficient box.

\subsection*{Notation}

We use the convention
\begin{equation}\label{eq:starred-sum-convention}
\sumstar_{a\bmod b}
=\sum_{\substack{a\bmod b\\\gcd(a,b)=1}}.
\end{equation}
When evaluating a local sign on such a class, we use its least positive representative, since the Jacobi symbols in the local factors have positive denominators.
The $p$-adic valuation is denoted by $v_p$, and we set $v_p(0)=+\infty$.
Unless otherwise stated, implied constants may depend on the degree \(d\) and on $\varepsilon$.

\subsection*{Acknowledgements}

The author thanks Tim Browning for suggesting this problem and for his guidance, Stephanie Chan for helpful discussions, and Julie Desjardins, Cec\'{\i}lia Salgado, Ronald van~Luijk, and Anthony V\'arilly-Alvarado for carefully reading the manuscript and offering valuable suggestions.

\section{Root number formulas}\label{sec:local}

\subsection{The \texorpdfstring{$j=0$}{j=0} formula}

For $n\in\ZZ\setminus\{0\}$, let $\rw(n)\in\{\pm1\}$ denote the root number of the elliptic curve $y^2=x^3+n$, and set $\rw(0)=0$ by convention.
The displayed equation need not be minimal; throughout, $\rw(n)$ denotes the global root number of the elliptic curve over $\QQ$ defined by this equation.
Write
\begin{equation}\label{eq:def-units-j0}
n=2^{v_2(n)}n_2=3^{v_3(n)}n_3,
\quad 2\nmid n_2,\quad 3\nmid n_3,
\end{equation}
and set
\begin{equation}\label{eq:def-omega23}
\begin{aligned}
\omega_2(n)&=\rw_2(n)\left(\frac{-1}{|n_2|}\right), \\
\omega_3(n)&=\rw_3(n)(-1)^{v_3(n)},
\end{aligned}
\end{equation}
where $\bigl(\frac{\cdot}{\cdot}\bigr)$ is the Jacobi symbol and $\rw_p(n)\in\{\pm1\}$ denotes the local root number at $p\in\{2,3\}$.
The explicit formulas \cite[Lemma~4.1]{VA} are:
\begin{align*}
\rw_{2}(n)&=
\begin{cases}
-1 & \text{if } v_{2}(n)\equiv 0 \text{ or } 2 \pmod 6,\\
-1 & \text{if } v_{2}(n)\equiv 1,3,4 \text{ or } 5 \pmod 6 \text{ and } n_{2}\equiv 3 \pmod 4,\\
+1 & \text{otherwise},
\end{cases}\\
\rw_{3}(n)&=
\begin{cases}
-1 & \text{if } v_{3}(n)\equiv 1 \text{ or } 2 \pmod 6 \text{ and } n_{3}\equiv 1 \pmod 3,\\
-1 & \text{if } v_{3}(n)\equiv 4 \text{ or } 5 \pmod 6 \text{ and } n_{3}\equiv 2 \pmod 3,\\
-1 & \text{if } v_{3}(n)\equiv 0 \pmod 6 \text{ and } n_{3}\equiv 5 \text{ or } 7 \pmod 9,\\
-1 & \text{if } v_{3}(n)\equiv 3 \pmod 6 \text{ and } n_{3}\equiv 2 \text{ or } 4 \pmod 9,\\
+1 & \text{otherwise}.
\end{cases}
\end{align*}
For primes $p\ge 5$, define
\begin{equation}\label{eq:def-omega-p}
\omega_p(n)=
\begin{cases}
1 & \text{if } v_p(n)\equiv 0,1,3,5 \pmod 6,\\[0.2em]
\left(\dfrac{-3}{p}\right) & \text{if } v_p(n)\equiv 2,4 \pmod 6.
\end{cases}
\end{equation}
V\'arilly-Alvarado \cite[Proposition~4.4]{VA} then shows that
\begin{equation}\label{eq:rootnumber-factor}
\rw(n)
=-\,\omega_2(n)\,\omega_3(n) \prod_{\substack{p^2\mid n\\ p\ge 5}} \omega_p(n).
\end{equation}
In particular, the contribution from primes $p\ge 5$ is multiplicative, but $\rw(n)$ itself is not.

\subsection{The \texorpdfstring{$j=1728$}{j=1728} formula}

Two corrections to \cite[Lemma~4.7 and Proposition~4.8]{VA} are required; they are described in Remark~\ref{rem:VA-corrections} after the corrected formulas.
Our derivation uses Rizzo's three local tables \cite[Tables~I--III]{Riz03}; for $p\ge5$ one may equivalently use Rohrlich's local formula \cite{Roh93}.
We use positive denominators in Jacobi symbols throughout.

For $n\in\ZZ\setminus\{0\}$, let $\rw^\star(n)\in\{\pm 1\}$ denote the root number of the elliptic curve
\[
E_n^\star:\quad y^2=x^3+n x,
\]
and set $\rw^\star(0)=0$ by convention.
Again, this equation is not assumed minimal.
Write
\begin{equation}\label{eq:def-odd-part-star}
n=2^{v_2(n)}n_2,\quad 2\nmid n_2.
\end{equation}
\begin{lemma}\label{lem:W23-1728}
Let $n\in\ZZ\setminus\{0\}$, with $n_2$ as in \eqref{eq:def-odd-part-star}.
Then the local root numbers $\rw_2^\star(n),\rw_3^\star(n)\in\{\pm1\}$ of $E_n^\star$ at $2$ and $3$ are given by
\begin{align*}
& \rw_2^\star(n)=
\begin{cases}
-1 & \text{if } v_2(n)\equiv 1 \pmod 4 \text{ and } n_2\equiv 1 \text{ or } 3 \pmod 8,\\
-1 & \text{if } v_2(n)\equiv 3 \pmod 4 \text{ and } n_2\equiv 5 \text{ or } 7 \pmod 8,\\
-1 & \text{if } v_2(n)\equiv 0 \pmod 4 \text{ and } n_2\equiv 1,5,9,11,13,15 \pmod{16},\\
-1 & \text{if } v_2(n)\equiv 2 \pmod 4 \text{ and } n_2\equiv 1,3,5,7,11,15 \pmod{16},\\
+1 & \text{otherwise},
\end{cases}
\\
& \rw_3^\star(n)=
\begin{cases}
-1 & \text{if } v_3(n)\equiv 2 \pmod 4,\\
+1 & \text{otherwise}.
\end{cases}
\end{align*}
\end{lemma}

\begin{proof}
For the displayed Weierstrass equation we have
\[
c_4=-48n,\quad c_6=0,\quad \Delta=-64n^3.
\]
Write $v_p(n)=4k+e$, where $0\le e<4$, and put $u=p^{-v_p(n)}n$.
Multiplication of $n$ by $p^4$ gives a $\QQ_p$-isomorphic curve via $(x,y)=(p^2X,p^3Y)$, so it is enough to read Rizzo's tables for the parameter $p^e u$.

At $p=3$ the four valuation triples are
\[
\begin{array}{|c|c|c|}
\hline
e &(v_3(c_4),v_3(c_6),v_3(\Delta))&\rw_3^\star(n)\\
\hline
0&(1,\infty,0)&+1\\
\hline
1&(2,\infty,3)&+1\\
\hline
2&(3,\infty,6)&-1\\
\hline
3&(4,\infty,9)&+1.\\
\hline
\end{array}
\]
These are respectively the rows $(1,\ge3,0)$, $(2,\ge5,3)$, $(3,\ge6,6)$, and $(4,\ge8,9)$ of \cite[Table~II]{Riz03}.
This proves the asserted $3$-adic formula.

At $p=2$ we have
\[
(v_2(c_4),v_2(c_6),v_2(\Delta))=(4+e,\infty,6+3e),
\quad
\frac{c_4}{2^{4+e}}=-3u,\quad
\frac{\Delta}{2^{6+3e}}=-u^3.
\]
Using Rizzo's convention of reducing an admissible triple modulo $(4,6,12)$, the relevant rows of \cite[Table~III]{Riz03} and their unit conditions simplify, because $c_6=0$, as follows:
\[
\begin{array}{|c|c|c|}
\hline
e&(v_2(c_4),v_2(c_6),v_2(\Delta))&\rw_2^\star(n)=+1\ \text{iff}\\
\hline
0&(4,\ge7,6)&u\equiv3,7\pmod {16}\\
\hline
1&(5,\ge9,9)&u\equiv5,7\pmod 8\\
\hline
2&(2,\ge4,0),\ (2,\ge5,0)&u\equiv9,13\pmod {16}\\
\hline
3&(3,\ge6,3)&u\equiv1,3\pmod 8.\\
\hline
\end{array}
\]
Taking complements among the odd residue classes gives exactly the four $2$-adic cases in the statement.
In particular, for $e=3$ the normalized unit is $c_4/2^7=-3u$; the row $(3,\ge6,3)$ gives sign $+1$ precisely when $-3u\equiv5$ or $7\pmod8$, equivalently when $u\equiv1$ or $3\pmod8$.
Thus the sign is $-1$ precisely for $u\equiv5$ or $7\pmod8$.
\end{proof}

\begin{definition}\label{def:omega23-1728}
For $n\in\ZZ\setminus\{0\}$ set
\[
\omega_2^\star(n)=\rw_2^\star(n)\left(\frac{-2}{|n_2|}\right),
\quad
\omega_3^\star(n)=\rw_3^\star(n).
\]
For primes $p\ge 5$, define
\begin{equation}\label{eq:def-omega-p-star}
\omega_p^\star(n)=
\begin{cases}
1 & \text{if } v_p(n)\not\equiv 2 \pmod 4,\\[0.2em]
\Big(\dfrac{-1}{p}\Big) & \text{if } v_p(n)\equiv 2 \pmod 4.
\end{cases}
\end{equation}
\end{definition}

\begin{proposition}\label{prop:rootnumber-factor-1728}
For $n\in\ZZ\setminus\{0\}$ we have the factorisation
\begin{equation*}
\rw^\star(n)
=
-\,\omega_2^\star(n)\,\omega_3^\star(n)\,
\prod_{\substack{p^2\mid n\\ p\ge 5}} \omega_p^\star(n).
\end{equation*}
\end{proposition}

\begin{proof}
For $p\ge5$, \cite[Table~I]{Riz03} (equivalently, Rohrlich's local formula \cite{Roh93}) shows that the local root number depends only on $v_p(n)\pmod4$.
More precisely,
\[
\rw_p^\star(n)=
\begin{cases}
1,&v_p(n)\equiv0\pmod4,\\
\bigl(\frac{-2}{p}\bigr),&v_p(n)\equiv1\text{ or }3\pmod4,\\
\bigl(\frac{-1}{p}\bigr),&v_p(n)\equiv2\pmod4.
\end{cases}
\]
Moreover $\bigl(\frac{-2}{3}\bigr)=1$, and hence
\[
\left(\frac{-2}{|n_2|}\right)
=\prod_{p\ge5}\left(\frac{-2}{p}\right)^{v_p(n)}.
\]
Combining the last two displays shows that the contribution of all primes $p\ge5$ is the Jacobi symbol absorbed into $\omega_2^\star(n)$, together with the extra factor $\bigl(\frac{-1}{p}\bigr)$ precisely when $v_p(n)\equiv2\pmod4$.
Multiplication by the local factors at $2$ and $3$, and by the archimedean factor $-1$, proves the proposition.
\end{proof}

\begin{remark}\label{rem:VA-corrections}
The formulas above differ from the corresponding statements in \cite{VA} in two places.
First, Rizzo's Table~III, row $(3,\ge6,3)$, gives $\rw_2^\star(8)=+1$, rather than the value stated in \cite[Lemma~4.7]{VA}.
Second, the formula in \cite[Proposition~4.8]{VA} omits the factor $\bigl(\frac{2}{p}\bigr)$ when $p\ge5$ and $v_p(n)\equiv1\pmod4$; for example, direct multiplication of the local factors gives $\rw^\star(5)=-1$, but the displayed formula there gives $+1$.
\end{remark}

\section{Root numbers in arithmetic progressions}\label{sec:average}

\subsection{The contribution from primes $p \ge 5$}

We begin with arguments, common to both root number families, based on divisor sums and on partitioning the integers according to their $2$- and $3$-adic valuations and unit residue classes.
For a bounded interval $I\subset\RR$, write
\begin{equation}\label{eq:interval-notation}
I_+=I\cap(0,\infty),\quad I_-=I\cap(-\infty,0],
\quad
\langle I\rangle=1+\sup\{|t|:t\in I\}.
\end{equation}
For a positive integer $d$, define
\begin{equation}\label{eq:def-Pplus}
P^+(1)=1,\quad P^+(d)=\max\{p:p\mid d\}\quad\text{for }d>1.
\end{equation}

\begin{lemma}\label{lem:squarefull-kernel}
Let $A:\NN\to\{\pm1\}$ be multiplicative, with $A(p)=1$ for every prime $p\ge5$ and $A(p^r)=1$ for $p\in\{2,3\}$ and $r\ge1$.
For $z\in[5,\infty]$ and $n\in\NN$ define
\[
A(n;z)=\sum_{\substack{d\mid n\\P^+(d)\le z}}(A*\mu)(d),\quad
C_A(u,q;z)=\sum_{\substack{d\ge1,\ P^+(d)\le z\\\gcd(d,q)\mid u}}
(A*\mu)(d)\frac{\gcd(d,q)}d.
\]
Extend $A(\,\cdot\,;z)$ evenly to nonzero integers and put $A(0;z)=0$.
Then
\begin{enumerate}
\item $A*\mu$ is multiplicative, supported on squarefull integers coprime to $6$, and $|(A*\mu)(p^r)|\le2$;
\item for every $\varepsilon>0$,
\[
\sum_{\substack{1\le n\le N\\n\equiv u\pmod q}}A(n;z)
=C_A(u,q;z)\frac Nq+O(N^{1/2+\varepsilon}),
\]
uniformly in $N,q,u,z$, and $|C_A(u,q;z)|\le1$;
\item if $q\mid Q$, every prime dividing $Q/q$ belongs to $\{2,3\}$, and $v\equiv u\pmod q$, then $C_A(v,Q;z)=C_A(u,q;z)$.
\end{enumerate}
\end{lemma}

\begin{proof}
Put $h_A=A*\mu$.
The identity $h_A(p^r)=A(p^r)-A(p^{r-1})$ proves the first assertion.
For the second, expand $A(n;z)$ and count the solutions of $dk\equiv u\pmod q$.
If $g=\gcd(d,q)$, this congruence is soluble precisely when $g\mid u$, and then the admissible $k$ occupy one class modulo $q/g$.
Hence
\[
\sum_{\substack{n\le N\\n\equiv u\pmod q}}A(n;z)
=\frac Nq\sum_{\substack{d\le N,\ P^+(d)\le z\\\gcd(d,q)\mid u}}
h_A(d)\frac{\gcd(d,q)}d
+O\left(\sum_{d\le N}|h_A(d)|\right).
\]
The squarefull support and the bound $|h_A(p^r)|\le2$ make the error $O(N^{1/2+\varepsilon})$.
Extending the main sum to $d<\infty$ has the same cost, since
\[
N\sum_{d>N}\frac{|h_A(d)|}{d}
\ll N^{1/2+\varepsilon}.
\]
This proves the asserted progression formula.
Multiplicativity and the convolution identity also give
\[
A(n;z)=\prod_{\substack{p\mid n\\p\le z}}A\bigl(p^{v_p(n)}\bigr)\in\{\pm1\}.
\]
The same divisor decomposition identifies $C_A(u,q;z)$ as the limiting mean of $A(\,\cdot\,;z)$ on the progression, whence $|C_A(u,q;z)|\le1$.
For (3), every $d$ with $(A*\mu)(d)\ne0$ is coprime to $Q/q$.
Thus $\gcd(d,Q)=\gcd(d,q)$, and the two divisibility conditions agree because $v\equiv u\pmod q$.
\end{proof}

\begin{proposition}\label{prop:cell-progressions}
Let $A:\NN\to\{\pm1\}$ be multiplicative, with $A(p)=1$ for every prime $p\ge5$ and $A(p^r)=1$ for $p\in\{2,3\}$ and $r\ge1$.
Fix positive integers $b_2,b_3$, set $T=2^{b_2}3^{b_3}$ and $M_{r,s}=2^{r+b_2}3^{s+b_3}$, and choose signs $\sigma_{r,s,t}$ for $r,s\ge0$ and $t\in(\ZZ/T\ZZ)^\times$.
For $n\ne0$, write
\[
r=v_2(n),\quad s=v_3(n),\quad
t\equiv \frac{n}{2^r3^s}\pmod T,
\]
and define
\[
w(n;z)=\sign(n)\sigma_{r,s,t}A(|n|;z),\quad w(0;z)=0.
\]
For $q\ge1$ and $u\pmod q$, put
\[
Q_{r,s}(q)=\operatorname{lcm}(q,M_{r,s})
\]
and
\[
\lambda_{r,s,t}(u,q)=
\begin{cases}
q/Q_{r,s}(q),&u\equiv2^r3^s t\pmod{\gcd(q,M_{r,s})},\\
0,&\text{otherwise},
\end{cases}
\]
and
\[
\rho_w(u,q)=\sum_{r,s\ge0}\ \sumstar_{t\bmod T}
\sigma_{r,s,t}\lambda_{r,s,t}(u,q).
\]
Then, for $z\in[5,\infty]$, every bounded interval $I$, and every $\varepsilon>0$,
\[
\sum_{\substack{n\in I\\n\equiv u\pmod q}}w(n;z)
=\rho_w(u,q)C_A(u,q;z)\frac{|I_+|-|I_-|}{q}
+O_{b_2,b_3}(\langle I\rangle^{1/2+\varepsilon}),
\]
uniformly in $q,u,I,z$.
Moreover $|\rho_w(u,q)|\le1$, and $\rho_w(u,q)$ depends on $u$ only modulo $2^{v_2(q)}3^{v_3(q)}$.
\end{proposition}

\begin{proof}
First take $I=(0,N]$ with $N\ge1$, and partition its integers by $(r,s,t)$.
When compatible, the congruences $n\equiv u\pmod q$ and $n\equiv2^r3^s t\pmod{M_{r,s}}$ determine one class modulo $Q_{r,s}(q)$.
Apply part~(2) of Lemma~\ref{lem:squarefull-kernel} with $\varepsilon/2$ in place of $\varepsilon$ to estimate the sum of $A(n;z)$ on this class.
Part~(3) then replaces its factor $C_A(v,Q_{r,s}(q);z)$ by $C_A(u,q;z)$.
Only $r,s\ll\log N$ occur up to $N$, so $(\log N)^2\ll N^{\varepsilon/2}$ absorbs the number of cells and the accumulated error is $O(N^{1/2+\varepsilon})$.
The omitted relative density after completing the main term to all valuation cells is
\[
O\bigl(\min\{1,2^{v_2(q)}/N\}+
\min\{1,3^{v_3(q)}/N\}\bigr).
\]
After multiplication by $N/q$, this contributes $O(1)$, since $2^{v_2(q)},3^{v_3(q)}\le q$.
This proves the positive initial segment formula.
The definition of $w$ also gives the negative initial segment formula with the opposite sign.
To see this explicitly, substitute $n=-m$.
A cell indexed by $(u,t)$ becomes the cell indexed by $(-u,-t)$, and
\[
C_A(-u,q;z)=C_A(u,q;z),\quad
\lambda_{r,s,-t}(-u,q)=\lambda_{r,s,t}(u,q),
\]
but the factor $\sign(n)$ contributes $-1$.
Taking differences of the positive and negative initial segment formulas at the endpoints of $I_+$ and $I_-$ proves the assertion for a general bounded interval $I$.

The $\lambda_{r,s,t}(u,q)$ are the nonnegative relative densities of disjoint valuation cells inside $u\pmod q$ and sum to $1$.
Hence $|\rho_w(u,q)|\le1$.
Their compatibility conditions involve $u$ only modulo $2^{v_2(q)}3^{v_3(q)}$, proving the final assertion.
\end{proof}

We now apply the argument to the $j=0$ family.

The next three quantities isolate the large prime contribution to the root number.
The function $a$ packages the local factors at primes $p\ge 5$, its M\"obius transform $h$ is the convenient squarefull kernel for divisor sums, and $c(u,q)$ is the mean value of $a$ on the progression $u\pmod q$.
The complementary factor $\rho(u,q)$, introduced in the next subsection, records the contribution from the primes $2$ and $3$.

For $n\in\ZZ\setminus\{0\}$ define the even function
\begin{equation}\label{eq:def-a}
a(n)=\prod_{\substack{p^2\mid n\\ p\ge 5}}\omega_p(n).
\end{equation}
Its restriction to $\NN$ is multiplicative.
Write
\begin{equation}\label{eq:a-conv}
a=1*h, \quad \text{ equivalently } h=a*\mu.
\end{equation}
Then $h$ is multiplicative, $h(p)=0$ for every prime $p$, and
\begin{equation}\label{eq:h-prime-powers}
h(p^r)=a(p^r)-a(p^{r-1}).
\end{equation}
Since \eqref{eq:def-a} gives $a(p^r)=1$ for all $r\ge 1$ when $p\in\{2,3\}$, equation \eqref{eq:h-prime-powers} yields $h(p^r)=0$ for all $r\ge 1$, and multiplicativity gives
\begin{equation}\label{eq:h-coprime-6}
h(d)=0, \quad \text{if }\gcd(d,6)>1.
\end{equation}
In particular, $h$ is supported on squarefull integers coprime to $6$.
For $q\ge 1$ and a residue class $u\pmod{q}$, set
\begin{equation}\label{eq:def-c}
c(u,q)=\sum_{\substack{d\ge 1\\ \gcd(d,q)\mid u}}h(d)\,\frac{\gcd(d,q)}{d}.
\end{equation}

\begin{lemma}\label{lem:c-local}
Let $q\ge 1$ and $u\pmod q$.
With $h$ as in \eqref{eq:a-conv} and $c(u,q)$ as in \eqref{eq:def-c}, we have the Euler product
\[
c(u,q)=\prod_{p}c_p\bigl(u\bmod p^{v_p(q)},p^{v_p(q)}\bigr),
\]
where, writing $k=v_p(q)$ and $t=v_p(u)$,
\begin{equation}\label{eq:cp-def}
c_p(u,p^k)=\sum_{\substack{r\ge 0\\ \min(r,k)\le t}}h(p^r)\,p^{\min(r,k)-r}.
\end{equation}
\end{lemma}

If $u$ denotes a residue class, any integer representative may be used: only $\min\{v_p(u),v_p(q)\}$ enters the formula, and this is independent of the representative.

\begin{proof}
The summand in \eqref{eq:def-c}, taken to be zero unless $\gcd(d,q)\mid u$, is multiplicative in $d$.
Hence $c(u,q)$ factors as an Euler product.
For a given prime $p$, write $k=v_p(q)$, so that $\gcd(p^r,q)=p^{\min(r,k)}$, and write $t=v_p(u)$.
Then $p^{\min(r,k)}\mid u$ is equivalent to $\min(r,k)\le t$, which gives \eqref{eq:cp-def}.

Absolute convergence follows because $h(p)=0$ and $|h(p^r)|\ll 1$ for $r\ge 2$, while for $r>k$ the factor $p^{\min(r,k)-r}=p^{k-r}$ decays geometrically, and for $r\le k$ there are only finitely many terms.
Globally, only finitely many primes divide $q$.
For $p\nmid q$, the sum of the absolute values of the non-constant local terms is $O(p^{-2})$.
The Euler product is therefore absolutely convergent.
\end{proof}

\begin{lemma}\label{lem:sum-a}
Let $\varepsilon>0$.
Then
\begin{equation}\label{eq:sum-a}
\sum_{\substack{1\le n\le N\\ n\equiv u\pmod q}} a(n)
=c(u,q) \, \frac{N}{q} +O\big(N^{1/2+\varepsilon}\big),
\end{equation}
uniformly in $N\ge 1$, $q\ge 1$ and $u\pmod q$.
\end{lemma}

\begin{proof}
Apply Lemma~\ref{lem:squarefull-kernel}(2) with $A=a$ and $z=\infty$.
Then $A(n;\infty)=a(n)$ and $C_A(u,q;\infty)=c(u,q)$.
\end{proof}

\subsection{Contribution of the primes $2$ and $3$}

For $r,s\ge 0$ set
\[
M_{r,s}=2^{r+2}3^{s+2},
\quad
Q_{r,s}(q)=\operatorname{lcm}(q,M_{r,s}).
\]
For $q\ge 1$, $u\pmod q$ and $t\bmod36$ with $\gcd(t,6)=1$ define
\begin{equation*}
\lambda_{r,s,t}(u,q)=
\begin{cases}
\dfrac{q}{Q_{r,s}(q)} & \text{if } u\equiv2^r3^s t\pmod{\gcd(q,M_{r,s})},\\
0 & \text{otherwise}.
\end{cases}
\end{equation*}
We then set
\begin{equation}\label{eq:def-rho}
\rho(u,q)=\sum_{r,s\ge 0}\ \sumstar_{t\bmod36}
\bigl(-\omega_2(2^{r}3^{s}t)\,\omega_3(2^{r}3^{s}t)\bigr)\,
\lambda_{r,s,t}(u,q),
\end{equation}
whose defining series converges absolutely.
Here and in every analogous sum, the local signs are evaluated at the least positive representative of $t$, as stipulated in the introduction.

For $z\in[5,\infty]$ and $n\ne0$, define
\begin{equation}\label{eq:def-ac-trunc}
a(n;z)=\sum_{\substack{d\mid n\\P^+(d)\le z}}h(d),\quad
c(u,q;z)=\sum_{\substack{d\ge1,\ P^+(d)\le z\\\gcd(d,q)\mid u}}
h(d)\frac{\gcd(d,q)}d,
\end{equation}
where the restrictions involving $P^+(d)$ are omitted when $z=\infty$; also set $a(0;z)=0$.
These are the truncations associated with $a$ in Lemma~\ref{lem:squarefull-kernel}.
Thus $c(u,q;\infty)=c(u,q)$, and Lemma~\ref{lem:c-local} gives
\[
c(u,q;z)=\prod_{p\le z}c_p\bigl(u,p^{v_p(q)}\bigr)
\]
when $z<\infty$.
Define
\begin{equation}\label{eq:rw-eta-def}
\rw(n;z)=-\omega_2(n)\omega_3(n)a(n;z)\quad\text{for }n\ne0,
\quad \rw(0;z)=0.
\end{equation}
Then $\rw(n;\infty)=\rw(n)$, and for finite $z$ this is the root number product truncated at $p\le z$.
In particular, for $\eta\ge5$ and $n\ne0$,
\begin{equation}\label{eq:a-eta-conv}
a(n;\eta)
=\prod_{5\le p\le\eta}\left(\sum_{r=0}^{v_p(n)}h(p^r)\right)
=\sum_{\substack{q\mid n\\P^+(q)\le\eta}}h(q).
\end{equation}

\begin{proposition}\label{prop:rw-interval}
Let $q\ge1$, let $u\pmod q$, and let $I\subset\RR$ be a bounded interval.
For $z\in[5,\infty]$ and every $\varepsilon>0$,
\[
\sum_{\substack{n\in I\\n\equiv u\pmod q}}\rw(n;z)
=\rho(u,q)c(u,q;z)\frac{|I_+|-|I_-|}{q}
+O\bigl(\langle I\rangle^{1/2+\varepsilon}\bigr),
\]
uniformly in $q,u,I,z$.
Moreover $|\rho(u,q)|\le1$, and $\rho(u,q)$ depends on $u$ only modulo $2^{v_2(q)}3^{v_3(q)}$.
\end{proposition}

\begin{proof}
Apply Proposition~\ref{prop:cell-progressions} with $A=a$, $(b_2,b_3)=(2,2)$, and
\[
\sigma_{r,s,t}
=-\omega_2(2^r3^s t)\omega_3(2^r3^s t).
\]
For positive $n$, a cell fixes $v_2(n)$, $v_3(n)$, the odd part modulo $4$, and the $3$-adic unit modulo $9$, hence fixes the local sign.
For negative $n$, the same cell fixes the two local root numbers; with $n_2$ as in \eqref{eq:def-units-j0},
\[
\left(\frac{-1}{|n_2|}\right)
=-\left(\frac{-1}{3^s t}\right),
\]
which supplies exactly the factor $\sign(n)$.
Proposition~\ref{prop:cell-progressions} now gives all the stated assertions.
\end{proof}

\subsection{Truncation error for $c(u,q)$}
We now compare the truncated Euler product $c(u,q; \eta)$ with the full factor $c(u,q)$.
Let $d \geq 2$ be a fixed integer.
Fix $x$ sufficiently large, and set
\[
\eta = \frac{\log x}{\log \log x}, \quad \theta = \frac{\log \log x}{\log \log \log x}.
\]
We take the lower bound on $x$ large enough that $\eta\ge5$ and $\theta>1$.
\begin{lemma}\label{lem:mert}
We have
\[
\sum_{\substack{\eta<p\le\log x \\ p~\text{prime}}}\frac1p=O(\theta^{-1}).
\]
\end{lemma}
\begin{proof}
Mertens' theorem and $\log\eta=\log\log x-\log\log\log x$ give the stated bound.
\end{proof}

The next proposition compares the full and truncated Euler products uniformly on the progressions used below.
\begin{proposition}\label{prop:diff}
Let $q\le x^d$ and $u\pmod q$, and suppose that no prime $p>\eta$ satisfies $p^2\mid\gcd(u,q)$.
Then
\[
|c(u,q)-c(u,q;\eta)|\ll\theta^{-1}.
\]
\end{proposition}
\begin{proof}
By Lemma~\ref{lem:c-local} and the definition of the truncation, set
\[
C_{\eta}(u,q)=\prod_{p>\eta}c_p\bigl(u\bmod p^{v_p(q)},p^{v_p(q)}\bigr),
\]
so that
\begin{equation}\label{eq:c-factor-tail}
c(u,q)-c(u,q; \eta)=c(u,q; \eta)\bigl(C_{\eta}(u,q)-1\bigr).
\end{equation}

From \eqref{eq:h-prime-powers} and $a(n)\in\{\pm1\}$ we obtain
\begin{equation}\label{eq:h-bound}
h(p)=0,\quad |h(p^r)|\le 2 \, \text{ for } \, r\ge 2.
\end{equation}
Fix a prime $p>\eta$, put $k=v_p(q)$ and $t=v_p(u)$, with $v_p(0)=+\infty$.
Using \eqref{eq:cp-def} and \eqref{eq:h-bound}, we obtain
\[
c_p(u,p^k)=1+\sum_{\substack{r\ge 2\\ \min(r,k)\le t}} h(p^r)\,p^{\min(r,k)-r}.
\]
The hypothesis gives $\min(k,t)\le1$.
Hence, on writing $\delta_p=c_p(u,p^k)-1$, the geometric series in the preceding display immediately yields
\begin{equation}\label{eq:delta-bound}
|\delta_p|
\ll
\begin{cases}
p^{-2}, & p\nmid q,\\
p^{-1}, & p\mid q,\ v_p(q)=1,\ p\mid u,\\
0, & \text{otherwise}.
\end{cases}
\end{equation}

From \eqref{eq:delta-bound} we obtain
\begin{align}
\nonumber \sum_{p>\eta}|\delta_p|
&\ll
\sum_{\substack{p>\eta\\ p\nmid q}} p^{-2}
\;+\;
\sum_{\substack{p>\eta\\ p\mid q}} p^{-1} \\
&\ll
\frac{1}{\eta}
\;+\;
\sum_{\substack{\eta<p\le \log x}} p^{-1}
\;+\;
\sum_{\substack{p>\log x\\ p\mid q}} p^{-1}. \label{eq:sum-delta-split2}
\end{align}
The first term is $O(\eta^{-1})$.
The middle term is $O(\theta^{-1})$ by Lemma~\ref{lem:mert}.
For the last term, write $p_1,\dots,p_R$ for the distinct prime divisors of $q$ exceeding $\log x$.
Then $(\log x)^R \le p_1\cdots p_R \le q \le x^d$, hence $R\le d\log x/\log\log x$, and therefore
\[
\sum_{\substack{p>\log x\\ p\mid q}} \frac1p \le \frac{R}{\log x}\ll \frac{d}{\log\log x}.
\]
Since $\eta^{-1}\ll (\log\log x)/\log x$ and $(\log\log x)^{-1}\ll \theta^{-1}$ for $x\ge 1000$, we deduce from \eqref{eq:sum-delta-split2} that
\begin{equation*}
\sum_{p>\eta}|\delta_p| \ll \theta^{-1}.
\end{equation*}

Absolute convergence and the standard product estimate therefore give
\begin{equation}\label{eq:Ctail-approx}
C_{\eta}(u,q)=1+O(\theta^{-1}).
\end{equation}
The divisor sum argument in Lemma~\ref{lem:squarefull-kernel}(2) identifies $c(u,q;\eta)$ with the limiting mean of $a(\cdot;\eta)\in\{\pm1\}$ on the residue class $u\pmod q$.
Hence $|c(u,q;\eta)|\le 1$.
From \eqref{eq:c-factor-tail} and \eqref{eq:Ctail-approx}, we obtain
\[
|c(u,q)-c(u,q; \eta)|
\le |c(u,q; \eta)| \cdot |C_{\eta}(u,q)-1|
\ll \theta^{-1}.
\]
\end{proof}

The same tail estimate holds for either family.

\begin{lemma}\label{lem:squarefull-tail}
Let $d\ge2$ and let $x$ be sufficiently large, and put
\[
\eta=\frac{\log x}{\log\log x},\qquad
\theta=\frac{\log\log x}{\log\log\log x}.
\]
Let $A:\NN\to\{\pm1\}$ be multiplicative, with $A(p)=1$ for every prime $p\ge5$ and $A(p^r)=1$ for $p\in\{2,3\}$ and $r\ge1$.
Let $q\le x^d$ and $u\pmod q$, and suppose that no prime $p>\eta$ satisfies $p^2\mid(u,q)$.
Then
\[
|C_A(u,q;\infty)-C_A(u,q;\eta)|\ll\theta^{-1}.
\]
\end{lemma}

\begin{proof}
The proof of Proposition~\ref{prop:diff} uses only $(A*\mu)(p)=0$, $|(A*\mu)(p^r)|\le2$, the Euler product deduced from Lemma~\ref{lem:squarefull-kernel}, and the bound $|C_A(u,q;\eta)|\le1$ from Lemma~\ref{lem:squarefull-kernel}(2).
Thus \eqref{eq:delta-bound}, the prime sum \eqref{eq:sum-delta-split2}, and the product inequality apply verbatim.
\end{proof}

\section{Averages for fixed forms}\label{sec:fixed-forms}

In this section we prove Theorem~\ref{thm:asy}(i).
The corresponding proof of Theorem~\ref{thm:asy}(ii), in which every odd-exponent irreducible factor has degree at most $6$, is given in Section~\ref{sec:j1728-generic}.
We use $S_{\mathrm{prim}}(x)$ as defined in \eqref{eq:def-Sprim} and the local factors $\omega_2,\omega_3,\omega_p$ from \eqref{eq:def-omega23} and \eqref{eq:def-omega-p}.
The proof relies on two auxiliary lemmas: Lemma~\ref{lem:asy-qadic} controls the contribution of pairs where a fixed prime divides the form to high order, and Lemma~\ref{lem:asy-large-squares} bounds the contribution from large prime squares.
Together with the lattice point asymptotic \eqref{eq:asy-periodic}, these allow the root number to be approximated by a periodic function on a set of full density.

We begin by recording the lattice point asymptotic we shall use.
If $\mathcal R\subset \RR^2$ is bounded with polygonal boundary and $v:\ZZ^2\to \CC$ is periodic modulo $M$, then
\begin{equation}\label{eq:asy-periodic}
\sum_{\substack{(m,n)\in N\mathcal R\cap \ZZ^2\\ \gcd(m,n)=1}}v(m,n)
=c_{\mathcal R,v}N^2+O_{\mathcal R,M,v}(N\log N),
\end{equation}
for some constant $c_{\mathcal R,v}\in \CC$.
Although the error term depends on the chosen period $M$, the leading constant depends only on $\mathcal R$ and $v$.
This follows by decomposing into residue classes modulo $M$, applying M\"obius inversion to the condition $\gcd(m,n)=1$, and using the standard lattice point estimate in a translate of a sublattice inside a polygonal region.

We isolate the approximation argument that is common to the two families.

\begin{lemma}\label{lem:periodic-approximation}
Let $D,P,R\in\ZZ[m,n]$ be nonzero homogeneous forms, where $P=1$ or $P$ is a squarefree product of primitive irreducible forms of degree at most $6$.
Let $\Sigma$ be a nonempty finite set of primes.
Suppose that a real-valued function $w$, bounded by $1$ on primitive pairs, has a factorisation
\begin{equation}\label{eq:periodic-template}
w(m,n)=\sign(R(m,n))b(m,n)
\prod_{\substack{p\notin\Sigma\\p\ge5}}u_p(m,n)
\end{equation}
whenever $D(m,n)P(m,n)R(m,n)\ne0$, with $b,u_p\in\{\pm1\}$.
Assume the following.
\begin{enumerate}
\item For every choice of integers $H_q\ge1$, indexed by $q\in\Sigma$, there is a periodic sign $b_H$, where $H=(H_q)_{q\in\Sigma}$, such that $b_H(m,n)=b(m,n)$ whenever $q^{H_q}\nmid D(m,n)$ for every $q\in\Sigma$.
\item For every $p\notin\Sigma$ with $p\ge5$, and every $K\ge1$, there is a periodic sign $u_{p,K}$ such that $u_{p,K}(m,n)=u_p(m,n)$ whenever $p^K\nmid P(m,n)$.
\item If $p\notin\Sigma$, $p\ge5$, and $p^2\nmid P(m,n)$, then $u_p(m,n)=1$.
\end{enumerate}
Then
\[
\lim_{N\to\infty}
\frac{1}{\#S_{\mathrm{prim}}(N)}
\sum_{(m,n)\in S_{\mathrm{prim}}(N)}w(m,n)
\]
exists and has absolute value at most $1$.
\end{lemma}

\begin{proof}
Fix $\delta>0$.
Lemma~\ref{lem:asy-qadic} supplies $H_q$ so that the union of the fixed prime exceptional sets has upper density less than $\delta$.
Choose $Y$ such that
\[
Y>\max\{\max\Sigma,Y_0(P)\}.
\]
For each prime $5\le p\le Y$ outside $\Sigma$, Lemma~\ref{lem:asy-qadic} supplies $K_p$ so that the set of pairs satisfying $p^{K_p}\mid P(m,n)$ for at least one such $p$ has upper density less than $\delta$.
On each of the finitely many polygonal sectors cut out by the real zeros of $R$, the product of $\sign(R)$, $b_H$, and the $u_{p,K_p}$ is periodic.
Equation~\eqref{eq:asy-periodic} therefore gives a limiting normalized sum for this periodic approximation.

The points with $D(m,n)P(m,n)R(m,n)=0$ contribute $O_{D,P,R}(N)$ and may be discarded.
Elsewhere the approximation can fail only on the two exceptional unions just described, or where $p^2\mid P(m,n)$ for some $p>Y$.
The last set has upper density $O_P(Y^{-1})$ by Lemma~\ref{lem:asy-large-squares}.
Consequently, the upper and lower limits of
\[
\frac{1}{N^2}\sum_{(m,n)\in S_{\mathrm{prim}}(N)}w(m,n)
\]
differ by $O(\delta+Y^{-1})$.
Letting first $\delta\to0$ and then $Y\to\infty$ proves existence.
Division by $\#S_{\mathrm{prim}}(N)=(6/\pi^2)N^2+O(N\log N)$ gives the stated average; its absolute value is at most $1$ because $|w|\le1$.
\end{proof}

\begin{lemma}\label{lem:asy-qadic}
Let $G\in \ZZ[x,y]\setminus\{0\}$ and let $q$ be prime.
Then
\[
\lim_{H\to\infty}\ \limsup_{N\to\infty}\frac{1}{N^2}
\#\Bigl\{1\le m,n\le N:\gcd(m,n)=1,\ q^H\mid G(m,n)\Bigr\}=0.
\]
\end{lemma}

\begin{proof}
Dropping the coprimality condition only enlarges the set, so it is enough to count pairs with $q^H\mid G(m,n)$.
This condition is periodic modulo $q^H$, hence its density in the box $[1,N]^2$ tends to the $\ZZ_q^2$-measure of
\[
E_H=\{(x,y)\in \ZZ_q^2:q^H\mid G(x,y)\}.
\]
The sets $E_H$ decrease to the $q$-adic zero set of $G$, which has measure $0$: writing $G(x,y)=\sum_j a_j(x)y^j$, the set of $x\in\ZZ_q$ for which all coefficients $a_j(x)$ vanish is finite, because at least one of the polynomials $a_j$ is nonzero.
For every other $x$, the fibre is the zero set of a nonzero polynomial in $y$, hence is finite and has $\ZZ_q$-measure zero.
Fubini's theorem finishes the argument.
\end{proof}

We also use the elementary implication
\begin{equation}\label{eq:asy-valuations}
\begin{gathered}
a,b\ne0,\quad
a\equiv b \pmod{q^{L+K}}, \quad v_q(a),v_q(b)\le L\\
\Longrightarrow\quad
v_q(a)=v_q(b), \quad
q^{-v_q(a)}a\equiv q^{-v_q(b)}b \pmod{q^K},
\end{gathered}
\end{equation}
obtained by factoring out $q^{\min(v_q(a),v_q(b))}$.

The following complementary estimate controls the pairs for which the square of some large prime divides the form value.

\begin{lemma}\label{lem:asy-large-squares}
Let $\Phi\in \ZZ[x,y]$ be either $1$, or a squarefree product of primitive irreducible homogeneous forms, each of degree at most $6$.
Then, for $Y\ge Y_0(\Phi)$,
\[
\limsup_{N\to\infty}\frac{1}{N^2}
\#\Bigl\{1\le m,n\le N:
\substack{\gcd(m,n)=1,\\
\exists p>Y\text{ with }p^2\mid\Phi(m,n)}
\Bigr\}
\ll_\Phi \frac{1}{Y}.
\]
\end{lemma}

\begin{proof}
The case $\Phi=1$ is empty, so assume $\Phi\ne1$.
Write $\Phi=\prod_{j=1}^t \phi_j$ with $\phi_j$ irreducible.
We enlarge $Y_0(\Phi)$ so that, for every prime $p>Y_0(\Phi)$, the reduction of the projective divisor $\{\Phi=0\}\subset\PP^1_{\FF_p}$ is squarefree.
For such a prime $p$, if $\gcd(m,n)=1$ and $p\mid\Phi(m,n)$, then $p$ divides at most one of the values $\phi_j(m,n)$.
Hence
\[
p^2\mid \Phi(m,n)\quad\Longrightarrow\quad
p^2\mid \phi_j(m,n)
\]
for some $j$.
It is therefore enough to treat a single factor $\phi=\phi_j$.
For each prime $p>Y$, let $R_p(N)$ denote the number of primitive pairs $1\le m,n\le N$ for which $p^2\mid\phi(m,n)$.
The squarefreeness assumption and Hensel lifting in the two affine charts of $\PP^1$ give, uniformly for $p>Y$,
\[
R_p(N)\ll \frac{N^2}{p^2}+N.
\]
Indeed, there are $O(1)$ simple projective solutions modulo $p^2$, and each solution direction imposes one congruence modulo $p^2$ in the relevant affine chart.
The contribution from primes in $(Y,N]$ is therefore at most
\[
\sum_{Y<p\le N}R_p(N)
\ll N^2\sum_{p>Y}\frac1{p^2}+N\pi(N)
\ll \frac{N^2}{Y}+o(N^2).
\]

It remains to consider $p>N$.
If $\deg\phi=1$, then $|\phi(m,n)|\le CN$ for some $C>0$.
For sufficiently large $N$, divisibility by $p^2$ therefore forces $\phi(m,n)=0$, which accounts for $O(N)$ pairs.
Suppose that $\deg\phi=2$.
Irreducibility excludes $\phi(m,n)=0$ at a primitive integer pair, while $|\phi(m,n)|\le CN^2$ shows that every relevant prime satisfies $p\le C'N$ for some $C'>0$.
Hence
\[
\sum_{N<p\le C'N}R_p(N)
\ll N^2\sum_{p>N}\frac1{p^2}+N\pi(C'N)
\ll \frac{N^2}{\log N}
=o(N^2).
\]
For $3\le\deg\phi\le6$, Greaves's large-prime estimates
\cite[Lemma~2 and the final large-prime estimate in the proof of the main theorem]{Gre92}
apply to the primitive irreducible form $\phi$.
Indeed, $\phi$ has nonzero projective discriminant, and after enlarging $Y_0(\Phi)$ every prime under consideration is a good-reduction prime for $\phi$.
The estimates give
\[
\#\{1\le m,n\le N:\exists~p>N,\ p^2\mid\phi(m,n)\}=o(N^2).
\]
The estimate is stated for all pairs in a box, so restricting to the positive box and to primitive pairs can only decrease the count.
Summing over the finitely many factors of $\Phi$ proves the lemma.
\end{proof}

\begin{lemma}\label{lem:chebotarev-projective}
Let $f\in\ZZ[z,w]$ be primitive and irreducible homogeneous, and let $K$ be a quadratic field.
Let $K_f$ be the residue field defined in \eqref{eq:def-Kh}.
If $K\not\subset K_f$, then there are infinitely many primes inert in $K$ for which the divisor $\{f=0\}$ has a simple point in $\PP^1(\FF_q)$.
\end{lemma}

\begin{proof}
If $f(T,1)$ is constant, then $K_f=\QQ$ and every prime outside a finite set gives the simple point $(1:0)$; choosing inert primes in $K$ finishes this case.
Otherwise choose a root $\alpha$ of $f(T,1)$, so that $K_f=\QQ(\alpha)$, let $L$ be the splitting field of $f(T,1)$, and put $E=LK$.
Set $H=\Gal(E/K_f)$, whose fixed field is $K_f$.
If every element of $H$ acted trivially on $K$, then $K$ would lie in $E^H=K_f$, contrary to the hypothesis.
Hence some $h\in H$ acts non-trivially on $K$.
It fixes $K_f$, and therefore fixes the root $\alpha$.
Every conjugate of $h$ is still non-trivial on the Galois quadratic field $K$ and fixes a conjugate of $\alpha$.
Chebotarev's theorem applied to this conjugacy class therefore gives infinitely many unramified primes that are inert in $K$ and for which the Frobenius action fixes a root of $f$.
Away from the finite set of primes of bad reduction, the fixed root reduces to a simple $\FF_q$-point of $\{f=0\}$.
\end{proof}

\begin{lemma}\label{lem:exact-valuation-classes}
Let $\Phi=\prod_j\phi_j$ be a squarefree product of primitive irreducible homogeneous forms, and let $q>\deg\Phi$ be a prime for which $\{\Phi=0\}\subset\PP^1_{\FF_q}$ is squarefree.
If $\{\phi_0=0\}$ has an $\FF_q$-point, then there are primitive residue classes $\mathcal C_0,\mathcal C_2$ modulo $q^3$ such that
\[
q\nmid\Phi(m,n)\quad\text{on }\mathcal C_0,
\]
and on $\mathcal C_2$ we have
\[
v_q(\phi_0(m,n))=2,\quad
q\nmid\phi_j(m,n)\quad\text{for }j\ne0.
\]
\end{lemma}

\begin{proof}
Since $q>\deg\Phi$, some point of $\PP^1(\FF_q)$ lies outside $\{\Phi=0\}$; any lift gives $\mathcal C_0$.
For $\mathcal C_2$, work in an affine chart containing the given point and let $g(T)$ be the dehomogenization of $\phi_0$.
Squarefreeness makes its root $u\bmod q$ simple and ensures that no other $\phi_j$ vanishes there.
Lift this root by Hensel's lemma to a root $v\bmod q^2$; thus $g(v)=q^2A$ for some $A\in\ZZ$.
For $s\bmod q$,
\[
g(v+sq^2)\equiv q^2\bigl(A+s g'(u)\bigr)\pmod{q^3}.
\]
Choose $s$ outside the unique class for which the right-hand side vanishes modulo $q^3$.
The resulting point has the required exact valuation, and its other factors remain units modulo $q$.
The same argument in the other affine chart covers the point at infinity.
\end{proof}

\newpage
\begin{lemma}\label{lem:squarefree-class-density}
Let $P\in\ZZ[x,y]$ be $1$ or a squarefree product of primitive irreducible homogeneous forms of degree at most $6$.
Let $M\ge1$ be divisible by every prime of nonsquarefree reduction for the projective divisor $\{P=0\}$.
Let $\mathcal C$ be a primitive residue class modulo $M$, represented by $(a,b)$ with $\gcd(a,b,M)=1$, and let $\mathcal R\subset(0,1]^2$ be a nonempty open polygonal region.
Then the primitive points in $N\mathcal R\cap\mathcal C$ for which
\[
p^2\nmid P(m,n)\quad\text{for every prime }p\nmid M
\]
have positive lower density relative to $N^2$.
\end{lemma}

\begin{proof}
For $p\nmid M$, let $\delta_p$ be the local density of the condition $p^2\nmid P(m,n)$ conditional on $p\nmid\gcd(m,n)$.
Put
\[
\rho_P(p)=\#\{z\in\PP^1(\FF_p):P(z)=0\}.
\]
Since $p\nmid M$, the projective divisor $\{P=0\}$ is squarefree modulo $p$.
Reduction from $\PP^1(\ZZ/p^2\ZZ)$ to $\PP^1(\FF_p)$ has fibres of size $p$, and Hensel lifting shows that above each simple zero of $P$ exactly one projective lift is a zero modulo $p^2$.
Consequently,
\[
\delta_p
=1-\frac{\rho_P(p)}{p(p+1)}
=1+O_P(p^{-2}),
\qquad
\delta_p>0.
\]

For a fixed $Y$, impose the displayed conditions only for $p\le Y$.
Equation~\eqref{eq:asy-periodic} gives a density
\[
c_{\mathcal R,M,\mathcal C}
\prod_{\substack{p\le Y\\p\nmid M}}\delta_p,
\]
where $c_{\mathcal R,M,\mathcal C}>0$ is the density of primitive points in $N\mathcal R\cap\mathcal C$.
The estimates for $\delta_p$ show that the infinite product $\prod_{p\nmid M}\delta_p$ converges to a positive number, so these finite-level densities are bounded below by a positive constant.
Lemma~\ref{lem:asy-large-squares} shows that the conditions with $p>Y$ remove at most $O_P(Y^{-1})N^2+o(N^2)$ points.
Taking $Y$ sufficiently large proves the claim.
\end{proof}

\subsection{The \texorpdfstring{$j=0$}{j=0} case}

We now prove Theorem~\ref{thm:asy}(i).
The key step is to decompose $\rw(F(m,n))$ using \eqref{eq:rootnumber-factor}, approximate the resulting factors by periodic functions away from a set of small density, and apply \eqref{eq:asy-periodic}.

\begin{proof}[Proof of Theorem~\ref{thm:asy}(i)]
Write
\[
F_{\mathrm{red}}=\prod_{i=1}^r f_i,\quad
P_F=\prod_{3\nmid e_i}f_i,
\quad
Q_F=\prod_{e_i\equiv2,4\pmod6}f_i,
\quad
R_F=\prod_{6\nmid e_i}f_i,
\]
with the convention that the empty product is $1$.

Let $\Sigma$ be the finite set of primes $p$ such that either $p\mid 6c$, or $\overline{F_{\mathrm{red}}}$ is not squarefree in $\FF_p[x,y]$.
For $p\notin \Sigma$ and $\gcd(m,n)=1$, at most one of the values $f_i(m,n)$ is divisible by $p$; otherwise $\bar f_i$ and $\bar f_j$ would have a common projective zero, contradicting the squarefreeness of $\overline{F_{\mathrm{red}}}$.
For $p\ge 5$ put
\[
u_p(m,n)=\omega_p(F(m,n))
\left(\frac{-3}{p}\right)^{-v_p(Q_F(m,n))}.
\]
If $p\notin \Sigma$, $\gcd(m,n)=1$, and $p^2\nmid P_F(m,n)$, then $u_p(m,n)=1$.
Indeed, if $p\nmid F_{\mathrm{red}}(m,n)$ then both valuations are $0$.
Otherwise exactly one $f_j(m,n)$ is divisible by $p$.
If $e_j\equiv0$ or $3\pmod6$, then $v_p(F(m,n))\equiv 0$ or $3 \pmod{6}$, so $\omega_p(F(m,n))=1$ and $v_p(Q_F(m,n))=0$.
If $e_j\equiv1$ or $5\pmod6$, then $p^2\nmid P_F(m,n)$ forces $v_p(f_j(m,n))=1$, hence $v_p(F(m,n))\equiv 1$ or $5 \pmod{6}$ and again $u_p(m,n)=1$.
If $e_j\equiv2$ or $4\pmod6$, then $v_p(F(m,n))\equiv 2$ or $4 \pmod{6}$ and $v_p(Q_F(m,n))=1$, so $u_p(m,n)=1$.

Let $\chi_3$ be the non-trivial character modulo $3$, and for $z\neq 0$ define
\[
z_3=z/3^{v_3(z)},
\quad
\psi(z)=\chi_3(z_3).
\]
Since $\chi_3(-1)=\chi_3(2)=-1$ and $\chi_3(p)=\bigl(\frac{-3}{p}\bigr)$ for $p\ge5$, we have
\begin{equation}\label{eq:asy-chi-id}
\prod_{p\ge 5}\left(\frac{-3}{p}\right)^{v_p(z)}
=\sign(z)(-1)^{v_2(z)}\psi(z)
\end{equation}
for every $z\neq 0$.

For primitive $(m,n)$ with $F(m,n)\neq 0$, define
\begin{align*}
b_F(m,n)
&=-\sign(c)\sign(F(m,n))\omega_2(F(m,n))\omega_3(F(m,n))\\
&\quad \times (-1)^{v_2(Q_F(m,n))}\psi(Q_F(m,n))
\prod_{\substack{q\in \Sigma\\ q\ge 5}}
\omega_q(F(m,n))
\left(\frac{-3}{q}\right)^{-v_q(Q_F(m,n))}.
\end{align*}
Since $\omega_p(t)=1$ unless $v_p(t)\equiv 2$ or $4 \pmod{6}$, extending the product in \eqref{eq:rootnumber-factor} to all primes $p\ge 5$ does not change it.
Multiplying and dividing by $\prod_{p\ge 5}\bigl(\frac{-3}{p}\bigr)^{v_p(Q_F(m,n))}$, and using \eqref{eq:asy-chi-id}, we obtain
\begin{equation}\label{eq:asy-factor-main}
\rw(F(m,n))
=\sign(R_F(m,n))b_F(m,n)
\prod_{\substack{p\notin \Sigma\\ p\ge 5}}u_p(m,n).
\end{equation}
Indeed, multiplication by $Q_F$ changes the parity precisely for the factors with $e_i\equiv2$ or $4\pmod6$.
Thus the exponent of $f_i$ in $FQ_F$ is odd exactly when $6\nmid e_i$, and $\sign(F(m,n))\sign(Q_F(m,n))=\sign(c)\sign(R_F(m,n))$.

We verify the hypotheses of Lemma~\ref{lem:periodic-approximation} with $D=F_{\mathrm{red}}$, $P=P_F$, $R=R_F$, $b=b_F$, and $u_p$ as above.
Given integers $H_q\ge1$ for $q\in\Sigma$, consider a pair satisfying $q^{H_q}\nmid F_{\mathrm{red}}(m,n)$ for every $q\in\Sigma$.
Then $v_q(f_i(m,n))\le H_q-1$ for all $i$, and hence
\[
\begin{aligned}
v_q(F(m,n))&\le L_q:=v_q(c)+\sum_i e_i(H_q-1),\\
v_q(Q_F(m,n))&\le L_q':=
\sum_{e_i\equiv2,4\pmod6}(H_q-1).
\end{aligned}
\]
Put
\[
B_2=\max(L_2+2,L_2'+1),\quad
B_3=\max(L_3+2,L_3'+1),
\]
and $B_q=1+\max(L_q,L_q')$ for $q\in\Sigma$, $q\ge5$.
If two primitive pairs outside the bad sets are congruent modulo
\[
M_H
=2^{B_2}3^{B_3}\prod_{\substack{q\in\Sigma\\q\ge5}}q^{B_q},
\]
then \eqref{eq:asy-valuations} fixes the relevant valuations of $F$ and $Q_F$, together with the unit parts modulo $4$ and $9$.
Thus every factor in $b_F$ is fixed, except apparently $\sign(F)$.
This last obstruction is illusory: for the signed odd part $z_2=z/2^{v_2(z)}$,
\[
\sign(z)\left(\frac{-1}{|z_2|}\right)=\chi_4(z_2),
\]
where $\chi_4$ is the primitive real character modulo $4$.
Consequently $b_F$ agrees off the bad sets with a periodic sign modulo $M_H$.

For $p\notin\Sigma$ and $K\ge1$, at pairs satisfying $p^K\nmid P_F(m,n)$, the exact valuation below $K$ of the unique relevant factor is fixed modulo $p^K$; factors absent from $P_F$ contribute the constant $1$.
Hence $u_p$ agrees there with a periodic sign.
Its value is $1$ when $p^2\nmid P_F(m,n)$, as verified above.
All three hypotheses of Lemma~\ref{lem:periodic-approximation} now hold, so the average $C_F$ exists and satisfies $|C_F|\le1$.
Since $F$ is nonzero, the pairs in $S_{\mathrm{prim}}(N)$ for which $F(m,n)=0$ number $O_F(N)$.
Away from this density-zero set, for each $\varepsilon\in\{\pm1\}$,
\[
\mathbf 1_{\{\rw(F(m,n))=\varepsilon\}}
=\frac{1+\varepsilon\rw(F(m,n))}{2}.
\]
Taking averages proves the asserted density formula.

Assume now that there exists $j_0$ with $3\nmid e_{j_0}$ and
\[
\QQ(\sqrt{-3})\not\subset K_{f_{j_0}}.
\]
Write $f=f_{j_0}$ and $e=e_{j_0}$.
By Lemma~\ref{lem:chebotarev-projective}, there are infinitely many primes $q\equiv 2\pmod 3$ for which $\{f=0\}$ has a simple point in $\PP^1(\FF_q)$.
Choose one such prime with $q\notin\Sigma$, $q\ge Y_0(P_F)$, and $q>\deg F_{\mathrm{red}}$.
Lemma~\ref{lem:exact-valuation-classes}, applied to $F_{\mathrm{red}}$ and $f$, gives two primitive classes modulo $q^3$: on the first $q\nmid F_{\mathrm{red}}(m,n)$, while on the second
\[
v_q(f(m,n))=2,\quad q\nmid f_j(m,n)\quad\text{for }j\ne j_0.
\]

For each $\ell\in\Sigma$, choose a primitive point $\mathbf a_\ell\in\ZZ_\ell^2$ with $F_{\mathrm{red}}(\mathbf a_\ell)\ne0$.
This is possible because $F_{\mathrm{red}}$ is a nonzero homogeneous form and hence does not vanish on all of $\PP^1(\QQ_\ell)$.
The valuations of $F(\mathbf a_\ell)$ and $Q_F(\mathbf a_\ell)$ are finite.
Choose $B_\ell$ large enough so that on the class $\mathbf a_\ell\bmod \ell^{B_\ell}$ the valuations and unit residues needed to determine the local factors in $b_F$ are fixed: at $\ell=2$ this includes the data determining the combined periodic factor $\sign(F)\omega_2(F)$ and $(-1)^{v_2(Q_F)}$, at $\ell=3$ the factors $\omega_3(F)$ and $\psi(Q_F)$, and at $\ell\ge5$ the factor
\[
\omega_\ell(F)
\left(\frac{-3}{\ell}\right)^{-v_\ell(Q_F)}.
\]
Let
\[
M_0=\prod_{\ell\in\Sigma}\ell^{B_\ell}.
\]
By the Chinese remainder theorem, these local classes define a primitive residue class modulo $M_0$.
Combining this common class with the two $q$-adic classes above gives primitive classes $\mathcal C_\mathcal A,\mathcal C_\mathcal B$ modulo $M=M_0q^3$.
They have identical local data at every prime dividing $M$ other than $q$.

Let $\mathcal R\subset(0,1]^2$ be a nonempty open polygonal region whose closure avoids the real zero set of $R_F$; on $\mathcal R$ the sign of $R_F$ is constant.
Every prime at which the projective divisor $\{P_F=0\}$ is not squarefree belongs to $\Sigma$, because $P_F$ is a subproduct of $F_{\mathrm{red}}$.
Hence the modulus $M$ is divisible by every such prime.
Apply Lemma~\ref{lem:squarefree-class-density} to $P_F$, $\mathcal R$, and each of the two classes modulo $M$.
We obtain subsets $\mathcal A_N,\mathcal B_N\subset S_{\mathrm{prim}}(N)\cap N\mathcal R$ of positive lower density on which $p^2\nmid P_F(m,n)$ for every $p\nmid M$.
After removing the $O_F(N)$ points for which $F(m,n)=0$, both sets still have positive lower density.

For primes $p\nmid M$, good reduction ensures that $p$ divides at most one value $f_i(m,n)$.
If there is no such value, then $u_p(m,n)=1$.
Otherwise, let $f_i(m,n)$ be the unique divisible value.
If $f_i$ occurs in $P_F$, then $p^2\nmid P_F(m,n)$ forces $v_p(f_i(m,n))=1$; if it does not, then $3\mid e_i$ and its contribution to $u_p$ is identically $1$.
Hence $u_p(m,n)=1$ in all cases.
Thus no prime outside $M$ changes the root number on either set.
The local factors at primes in $M_0$ agree by construction, and $\sign(R_F)$ is constant on $\mathcal R$.
On $\mathcal A_N$, we have $q\nmid F_{\mathrm{red}}(m,n)$, so $u_q(m,n)=1$.
On $\mathcal B_N$, $v_q(f(m,n))=2$.
Hence
\[
v_q(F(m,n))=2e\equiv 2\text{ or }4\pmod 6,
\quad
\omega_q(F(m,n))=\left(\frac{-3}{q}\right)=-1.
\]
If $e\equiv1$ or $5\pmod6$, then $v_q(Q_F(m,n))=0$; if $e\equiv2$ or $4\pmod6$, then $v_q(Q_F(m,n))=2$.
In either case
\[
u_q(m,n)=\omega_q(F(m,n))
\left(\frac{-3}{q}\right)^{-v_q(Q_F(m,n))}=-1.
\]
Hence the root numbers on $\mathcal A_N$ and $\mathcal B_N$ are constant and opposite.
Since both sets have positive lower density relative to $N^2$, and $\#S_{\mathrm{prim}}(N)\sim(6/\pi^2)N^2$, both signs occur with positive relative density inside $S_{\mathrm{prim}}(N)$.
The existing average therefore satisfies $|C_F|<1$.
\end{proof}

\begin{proof}[Proof of Corollary~\ref{cor:conditional-density-j0}]
Recall that $X_F$ is the surface defined in \eqref{eq:DP1-family}.
Separability makes every exponent in the factorisation of $F$ equal to $1$.
Thus Theorem~\ref{thm:asy}(i) applies and gives $|C_F|<1$.
The zero set $F(m,n)=0$ contributes $O_F(x)$ primitive pairs, so it has density $0$ in $S_{\mathrm{prim}}(x)$.
It follows that a positive proportion of the primitive pairs give smooth fibres with root number $-1$.
Positive primitive pairs represent distinct points of $\PP^1(\QQ)$, so these include infinitely many distinct fibres.

Under the stated finiteness hypothesis, the parity theorem gives positive Mordell--Weil rank, and hence infinitely many rational points, on every such fibre.
The union of the rational points on these infinitely many positive rank fibres is Zariski dense in the elliptic surface obtained by blowing up the anticanonical base point.
Indeed, a proper closed subset has only finitely many vertical components, while each of its horizontal components meets a given smooth fibre in only finitely many points.
Hence it cannot contain all rational points of infinitely many positive rank fibres.
The contraction to $X_F$ preserves density.
This is the argument of \cite[Remark~4.6]{VA}.
\end{proof}

\section{A transference principle for polynomial values}\label{sec:bst}

The theorem of Browning, Sofos and Ter\"av\"ainen \cite[Theorem~2.2]{BST} transfers estimates in arithmetic progressions for a divisor-bounded arithmetic function to the values of almost all polynomials of fixed degree ordered by height.
They work under logarithmic-power hypotheses, allowing a sparser set of exceptional moduli and imposing a restriction on $\gcd(u,q)$.
In the present application the available saving is only $\theta^{-1}$, where
\[
\theta=\frac{\log\log x}{\log\log\log x},
\]
and the obstruction is instead residue-class dependent, since the progression estimate may fail when $\gcd(u,q)$ is divisible by a large square of a prime.
Thus \cite[Theorem~2.2]{BST} does not apply verbatim and we require the following variant.

\begin{theorem}\label{thm:transference}
Let $d\ge2$ and let $\varepsilon>0$.
Suppose that $H\ge H_0(d,\varepsilon)$, and set $x=H^{1/(4d)}$.
Let $f:\ZZ\to\CC$ be a function satisfying the following conditions.
\begin{enumerate}
\item[\rm (B)] $|f(n)|\ll1$ for every $n\in\ZZ$;
\item[\rm (AP)] with
\[
\eta=\frac{\log x}{\log\log x},
\quad
\theta=\frac{\log\log x}{\log\log\log x},
\]
for every $1\le q\le x^d$ and every $u\pmod q$ such that no prime $p>\eta$ satisfies $p^2\mid\gcd(u,q)$, we have
\begin{align}\label{eq:APeta}
\sup_{\substack{I \text{ interval} \\ |I| \ge H^{1 - \varepsilon} \\
I \subset [-(d+3)Hx^d,(d+3)Hx^d]}}
\frac{q}{|I|}
\Bigg|
\sum_{\substack{n \in I \\ n \equiv u \pmod q}} f(n)
\Bigg| \, \ll \, \theta^{-1}.
\end{align}
\end{enumerate}
The function $f$ may depend on $H$, provided that the constants in these hypotheses are uniform in $H$.
Then, for any polynomial $g\in\ZZ[t]$ of degree at most $d$ with coefficients in $[-H,H]$ and $g(0)\ne0$, we have
\begin{align}\label{eq:transference-second-moment}
\sum_{|a|,|b| \leq H}
\Bigg|
\sum_{1 \leq n \leq x} f\big(a n+b n^2+g(n)\big)
\Bigg|^2 \ll \theta^{-1} H^2 x^2.
\end{align}
The implied constant may depend on $d$, $\varepsilon$, and the constants appearing in \textup{(B)} and \textup{(AP)}.
\end{theorem}

\begin{lemma}\label{lem:pair-exception-count}
Let $d \geq 2$, let $H$ be sufficiently large, and set $x=H^{1/(4d)}$, $\eta=\log x/\log\log x$, and $\theta=\log\log x/\log\log\log x$.
Let $g \in \ZZ[t]$ be of degree at most $d$ with coefficients in $[-H, H]$ and $g(0) \neq 0$.
Define
\begin{align}\label{eq:def-pair-set}
\mathcal{N} = \left\{(n_1, n_2) \in (\NN \cap [1,x])^2: 
\begin{array}{l}
\mathrm{(i)}\ n_1, n_2, |n_1 - n_2| > x/ \eta, \\
\mathrm{(ii)}\ \gcd(n_1, n_2) \le \eta, \\
\mathrm{(iii)}\ \text{no prime } p > \eta \text{ with } p \mid n_1 \\
\quad\text{and } p^2 \mid g(n_1)
\end{array}
\right\}.
\end{align}
Then
\begin{align}\label{eq:pair-exception-bound}
\#\big((\NN \cap [1,x])^2 \setminus \mathcal{N}\big) \ll \theta^{-1} x^2.
\end{align}
\end{lemma}

\begin{proof}
Pairs failing \textup{(i)} number $O(x^2/\eta)$, while those failing \textup{(ii)} number
\[
\ll x^2\sum_{m>\eta}\frac1{m^2}
\ll\eta^{-1}x^2.
\]
For \textup{(iii)}, write $g(t)=\sum_{i=0}^{d}c_i t^i$.
If $n_1=pm$, then
\[
g(n_1)\equiv c_0+pc_1m\pmod{p^2}.
\]
Thus $p^2\mid g(n_1)$ forces $p\mid c_0$, and for each such $p$ there are
\[
\#\{n_1\le x:p\mid n_1,\ p^2\mid g(n_1)\}\ll \frac{x}{p}+1.
\]
Put $\mathcal{P}=\{p>\eta:p\mid c_0\}$.
Since $\eta^{|\mathcal{P}|} \le |c_0| \le H=x^{4d}$, we have $|\mathcal{P}|\ll \log x/\log\log x$.
Splitting the reciprocal sum at $\log x$ and using Lemma~\ref{lem:mert} gives
\[
\sum_{p\in\mathcal{P}}\frac1p\ll\theta^{-1},
\qquad
|\mathcal P|\ll\frac{\log x}{\log\log x}=o(x/\theta).
\]
Hence $O(x/\theta)$ values of $n_1$, and therefore $O(x^2/\theta)$ pairs, fail \textup{(iii)}.
Since $\eta^{-1}\ll\theta^{-1}$, the result follows.
\end{proof}

\subsection{Proof of Theorem~\ref{thm:transference}}

\begin{proof}
Fix the polynomial $g$.
Expanding the square gives
\begin{align}\label{eq:second-moment-expanded}
\sum_{|a|,|b|\le H}
\left|\sum_{1\le n\le x} f(an+bn^2+g(n))\right|^2
&=\sum_{1\le n_1,n_2\le x}
\sum_{|a|,|b|\le H}
f(m_{n_1})\,\overline{f(m_{n_2})},
\end{align}
where $m_n=an+bn^2+g(n)$.
We bound the diagonal contribution $n_1=n_2$ and the off-diagonal contribution $n_1\ne n_2$ separately.

\noindent\emph{Diagonal terms.} The bound $|f|\ll1$ shows that the diagonal contribution is $O(H^2x)$.
Since $\theta=o(x)$, this is negligible compared with $\theta^{-1}H^2x^2$.

\noindent\emph{Off-diagonal terms.} Let $n_1\ne n_2$ and put $\Delta=n_1-n_2$.
Group the sum over $(a,b)$ according to
\[
m_2=a+bn_2.
\]
For a fixed $m_2$, the admissible values of $b$ are the integers satisfying $|b|\le H$ and $|m_2-bn_2|\le H$; they lie in an interval of length at most $2H/n_2$.
Every relevant $m_2$ satisfies
\[
|m_2|=|a+bn_2|\le H(1+n_2),
\]
so there are $O(Hn_2)$ possible values.
Since
\[
a+bn_1=m_2+b\Delta,
\]
the corresponding arguments of $f$ form one residue class in an interval $J_{m_2}$.
The bound $|f|\ll1$ therefore gives
\[
\left|\sum_{|a|,|b|\le H}
f(m_{n_1})\overline{f(m_{n_2})}\right|
\ll\sum_{m_2\,\mathrm{relevant}}|S_{m_2}|,
\]
where
\begin{equation}\label{eq:def-slice-sum}
S_{m_2}=
\sum_{\substack{m\in J_{m_2}\\ m\equiv u\pmod q}} f(m),
\end{equation}
and here
\[
q=n_1|\Delta|,
\quad
u\equiv n_1m_2+g(n_1)\pmod q,
\]
while
\begin{equation}\label{eq:slice-length}
|J_{m_2}|\le \frac{2n_1|\Delta|}{n_2}\,H.
\end{equation}
Since $n_1,n_2\le x$, we have $q=n_1|\Delta|\le x^2\le x^d$.

Finally, note that for $n\le x$ and $|a|,|b|\le H$ we have
\[
|an+bn^2+g(n)|
\le 2Hx^d+(d+1)Hx^d
=(d+3)Hx^d.
\]
Thus all intervals $J_{m_2}$ lie in the range where \eqref{eq:APeta} is available.

\noindent\emph{Pairs in $\mathcal N$.} Assume now that $(n_1,n_2)\in\mathcal N$ and $n_1\ne n_2$.
From \eqref{eq:slice-length} and the definition of $q$,
\begin{equation}\label{eq:slice-density}
\frac{|J_{m_2}|}{q}
\ll \frac{H}{n_2}.
\end{equation}

We call the pair $(u,q)$ in \eqref{eq:def-slice-sum} \emph{admissible} if no prime $p>\eta$ satisfies $p^2\mid \gcd(u,q)$.
We first dispose of the short slices
\begin{equation}\label{eq:short-slice}
|J_{m_2}|<H^{1-\varepsilon}.
\end{equation}
The trivial estimate for an arithmetic progression gives
\[
|S_{m_2}|
\ll 1+\frac{H^{1-\varepsilon}}q.
\]
Condition \textup{(i)} gives
\[
\frac{n_2}{q}
=\frac{n_2}{n_1|\Delta|}
\ll\frac{\eta^2}{x}.
\]
There are $O(Hn_2)$ relevant values of $m_2$, and hence all the short slices for the fixed pair $(n_1,n_2)$ contribute
\[
\ll Hn_2+H^{2-\varepsilon}\frac{n_2}{q}
\ll Hx+H^{2-\varepsilon}\frac{\eta^2}{x}
\ll \theta^{-1}H^2
\]
for sufficiently large $H$.

We now restrict to the slices with $|J_{m_2}|\ge H^{1-\varepsilon}$.
If $(u,q)$ is admissible, then the hypothesis \eqref{eq:APeta} applies to $S_{m_2}$, and gives
\begin{equation}\label{eq:good-slice-bound}
|S_{m_2}|
\ll \theta^{-1}\cdot \frac{|J_{m_2}|}{q}
\ll \theta^{-1}\cdot \frac{H}{n_2}.
\end{equation}

It remains to control the number of inadmissible $m_2$.
Suppose that for some prime $p>\eta$ we have $p^2\mid\gcd(u,q)$.
If $p^2\mid n_1$, then $p^2\mid u$ forces $p^2\mid g(n_1)$, contrary to condition~\textup{(iii)}.
If instead $v_p(n_1)=1$, then $p^2\mid q=n_1|\Delta|$ forces $p\mid\Delta$, and hence $p\mid n_2$, contrary to $\gcd(n_1,n_2)\le\eta<p$.
Thus $p\nmid n_1$.
Consequently, $p^2\mid \gcd(u,q)$ implies $p^2\mid\Delta$ and $p^2\mid u$.
Using $u=n_1m_2+g(n_1)$, this forces
\begin{equation}\label{eq:bad-slice-congruence}
n_1m_2 \equiv -g(n_1)\pmod{p^2}.
\end{equation}
For each such $p$, the congruence \eqref{eq:bad-slice-congruence} specifies one residue class modulo $p^2$, because $p\nmid n_1$.
The possible values of $m_2$ lie in an interval of length $O(Hn_2)$.
The relevant primes satisfy $p^2\mid \Delta$, so $p^2\le |\Delta|\le x$; since $n_2>x/\eta$, we have
\[
\frac{Hn_2}{p^2}\ge \frac{H}{\eta}.
\]
Thus one residue class modulo $p^2$ contains $O(Hn_2/p^2)$ possible values of $m_2$.
Summing over the primes $p>\eta$ with $p^2\mid\Delta$ gives
\[
Hn_2 \sum_{\substack{p>\eta\\p^2\mid\Delta}}\frac1{p^2}
\ll \frac{Hn_2}{\eta}
\]
inadmissible values of $m_2$.

For these exceptional $m_2$, counting the terms of the progression and using \eqref{eq:slice-density} gives
\[
|S_{m_2}|\ll 1+\frac{H}{n_2}.
\]
Hence the total contribution from inadmissible $m_2$ is
\[
\ll \frac{H n_2}{\eta}\left(1+\frac{H}{n_2}\right)
\ll \frac{H n_2}{\eta}+\frac{H^2}{\eta}.
\]
Since $n_2\le x\le H$ and $\eta^{-1}=o(\theta^{-1})$, this contribution is $\ll \theta^{-1}H^2$ per pair.

There are $O(Hn_2)$ admissible values of $m_2$, so \eqref{eq:good-slice-bound} contributes $O(\theta^{-1}H^2)$ per pair.

Summing over the $\ll x^2$ pairs in $\mathcal N$ therefore contributes $\ll \theta^{-1}H^2x^2$ to \eqref{eq:second-moment-expanded}.

\noindent\emph{Off-diagonal pairs not in $\mathcal{N}$.} By Lemma~\ref{lem:pair-exception-count}, the off-diagonal pairs in $(\NN\cap[1,x])^2\setminus \mathcal N$ number $\ll \theta^{-1}x^2$.
For each such pair the bound $|f|\ll1$ gives $O(H^2)$ for the inner sum in \eqref{eq:second-moment-expanded}, and hence a total contribution $\ll \theta^{-1}H^2x^2$.

Combining the diagonal bound with the off-diagonal bounds completes the proof of \eqref{eq:transference-second-moment}.
\end{proof}

Recall the truncated root number $\rw(\cdot;\eta)$ from \eqref{eq:rw-eta-def}.
\begin{lemma}\label{lem:j0-trunc-AP}
For each fixed $d\ge2$ and every fixed $0<\varepsilon<1/18$, the function $n\mapsto\rw(n)-\rw(n;\eta)$ satisfies hypotheses \textup{(B)} and \textup{(AP)} of Theorem~\ref{thm:transference}.
\end{lemma}

\begin{proof}
Hypothesis \textup{(B)} is immediate.
Recall $c(u,q;z)$ and $\rho(u,q)$ from \eqref{eq:def-ac-trunc} and \eqref{eq:def-rho}, and the interval notation from \eqref{eq:interval-notation}.
Hypothesis \textup{(AP)} follows from Proposition~\ref{prop:rw-interval}, used with $z=\infty$ and $z=\eta$, together with Proposition~\ref{prop:diff}.
Subtracting the two formulas leaves the main term
\[
\rho(u,q)\bigl(c(u,q)-c(u,q;\eta)\bigr)
\frac{|I_+|-|I_-|}{q},
\]
which is $O(\theta^{-1}|I|/q)$ under the admissibility condition on $(u,q)$.
For the remaining error, apply the interval estimate with exponent $\varepsilon$.
Since $q\le x^d$, $|I|\ge H^{1-\varepsilon}$, and $H=x^{4d}$, its normalized contribution is
\[
\ll x^d\,\frac{(Hx^d)^{1/2+\varepsilon}}{H^{1-\varepsilon}}
=x^{-d(1/2-9\varepsilon)}=o(\theta^{-1}).
\]
This proves \textup{(AP)}.
\end{proof}

Recall $g_\cc$ and $S(H)$ from \eqref{eq:def-gc} and \eqref{eq:def-SH}.
\begin{corollary}\label{cor:pointwise-trunc}
There exists a subset $\mathcal E \subset S(H)$ with
\[
\#\mathcal E \ll \theta^{-\frac{1}{3}}\,H^{d+1},
\]
such that for every $\cc\in S(H)\setminus\mathcal E$, we have
\[
\left|
\sum_{1\le n\le x}
\Bigl(\rw\bigl(g_{\cc}(n)\bigr)-\rw\bigl(g_{\cc}(n); \eta\bigr)\Bigr)
\right|
\ll \theta^{-\frac{1}{3}}\,x.
\]
\end{corollary}

\begin{proof}
Apply Theorem~\ref{thm:transference} and Lemma~\ref{lem:j0-trunc-AP}.
Fix the coefficients $c_i$ with $i\notin\{1,2\}$ and apply Theorem~\ref{thm:transference} with
\[
g(n)=\sum_{\substack{0\le i\le d\\i\notin\{1,2\}}}c_i n^i.
\]
Its constant term is $c_0\ne0$.
Summing the resulting estimate over the $O(H^{d-1})$ choices of the fixed coefficients gives
\[
\sum_{\cc\in S(H)}
\left|\sum_{1\le n\le x}
\bigl(\rw(g_\cc(n))-\rw(g_\cc(n);\eta)\bigr)\right|^{2}
\ll \theta^{-1}\,H^{d+1}\,x^{2}.
\]
Chebyshev's inequality, with threshold $\theta^{-1/3}x$, now gives
\[
\#\mathcal E
\ll \theta^{-1/3}H^{d+1},
\]
and the asserted estimate holds outside $\mathcal E$.
\end{proof}

\section{Root numbers at polynomial values}\label{sec:poly-sum}

Let $d\ge 2$ and take $H$ sufficiently large in terms of $d$.
We recall the parameters
\[
x=H^{1/(4d)},\quad 
\eta=\frac{\log x}{\log\log x},\quad
\theta = \frac{\log \log x}{\log \log \log x}.
\]
The proof separates the primes $p\ge 5$ from the primes $2$ and $3$.
The quantities $\nu_\cc(q)$ control how often $q$ divides $g_\cc(n)$ and lead to the local factors $\kappa_p(\cc)$ for $p\ge 5$, while $\kappa_{2,3}(\cc)$ collects the averaged contribution of the $2$- and $3$-adic signs.
The truncated asymptotic is obtained by combining these two pieces.

\subsection{Generic coefficient vectors}

Recall $S(H)$ and the polynomial $g_\cc$ from \eqref{eq:def-SH} and \eqref{eq:def-gc}.
We write $\cont(g_\cc)$ for its content and $\Delta(g_\cc)$ for its discriminant.
For $q\ge 1$ we define
\begin{equation}\label{eq:def-nu}
\nu_{\cc}(q)=\#\{u\bmod q:g_\cc(u)\equiv0\pmod q\},
\end{equation}
which is multiplicative in $q$.
The following are the standard bounds for $\nu_\cc(q)$.
Parts~(1) and (4) are \cite[Lemma~6.1]{BST}.
Part~(2) follows from the stronger bound there, since
$v_p(\cont(g_\cc))\le v_p(\Delta(g_\cc))$, and part~(3) is \cite[Eq.~(43)]{Ste}.

\begin{lemma}\label{lem:nu-bounds}
Assume that $g_\cc$ is separable of degree $d\ge 2$.
\begin{enumerate}
\item If $v_p(\cont(g_\cc))\ge k$ then $\nu_\cc(p^k)=p^k$.
\item If $v_p(\cont(g_\cc))<k$ then $\nu_\cc(p^k)\le d\,p^{k(1-1/d)+v_p(\Delta(g_\cc))/d}$.
\item If $p\nmid \cont(g_\cc)$, then $\nu_\cc(p^k)\le 2p^{v_p(\Delta(g_\cc))/2}+d-2$.
\item If $p\nmid \Delta(g_\cc)$, then $\nu_\cc(p^k)=\nu_\cc(p)\le d$.
\end{enumerate}
\end{lemma}

Define
\begin{equation}\label{eq:def-generic-arith}
\begin{aligned}
\mathcal G_{\mathrm{arith}}(H)=\bigl\{\cc\in S(H):{}&c_d\ne0,
\ \Delta(g_\cc)\ne0,\ \cont(g_\cc)\le\theta,\\
&v_p(\Delta(g_\cc))\le\theta\ \text{for every }p\le\eta\bigr\}.
\end{aligned}
\end{equation}
Whenever $\cc\in S(H)$ and $c_d\Delta(g_\cc)\ne0$, let $L_\cc/\QQ$ be the splitting field of $g_\cc$.
After labelling its $d$ roots, regard $\Gal(L_\cc/\QQ)$ as a subgroup of $S_d$.
The generic sets used for the two families are
\begin{equation}\label{eq:def-Gfamilies}
\begin{aligned}
\mathcal G_0(H)
&=\bigl\{\cc\in\mathcal G_{\mathrm{arith}}(H):
  \Gal(L_\cc/\QQ)=S_d,\ \QQ(\sqrt{-3})\not\subset L_\cc\bigr\},\\
\mathcal G_{1728}(H)
&=\bigl\{\cc\in\mathcal G_{\mathrm{arith}}(H):
  \Gal(L_\cc/\QQ)=S_d,\ \QQ(i)\not\subset L_\cc\bigr\}.
\end{aligned}
\end{equation}
Membership in either set implies that $g_\cc$ is irreducible over $\QQ$; in particular, $g_\cc(n)\ne0$ for every $n\in\ZZ$.

\begin{lemma}\label{lem:generic-families}
For $j\in\{0,1728\}$,
\[
\#\bigl(S(H)\setminus\mathcal G_j(H)\bigr)
\ll\theta^{-d}H^{d+1}.
\]
\end{lemma}

\begin{proof}
Fix $j\in\{0,1728\}$, and let $K$ be the quadratic field excluded in the definition of $\mathcal G_j(H)$.
Thus $K=\QQ(\sqrt{-3})$ if $j=0$, and $K=\QQ(i)$ if $j=1728$.
Write $K=\QQ(\sqrt D)$, with $D=-3$ or $D=-1$, respectively.
Define
\[
\begin{aligned}
E_1&=\{\cc\in S(H):c_d=0\ \text{or}\
   (c_d\Delta(g_\cc)\ne0\ \text{and}\ \Gal(L_\cc/\QQ)\ne S_d)\},\\
E_2&=\{\cc\in S(H):\Delta(g_\cc)=0\},\\
E_3&=\{\cc\in S(H):\cont(g_\cc)>\theta\},\\
E_4&=\{\cc\in S(H):\exists\,p\le\eta
   \text{ with }v_p(\Delta(g_\cc))>\theta\},\\
E_5&=\{\cc\in S(H):c_d\Delta(g_\cc)\ne0,\
   \Gal(L_\cc/\QQ)=S_d,\ K\subset L_\cc\}.
\end{aligned}
\]
Then $S(H)\setminus\mathcal G_j(H)\subseteq E_1\cup\cdots\cup E_5$.

Apply \cite[Corollary~1]{CD} to
\[
\mathcal G(X;T_0,\ldots,T_d)=T_dX^d+\cdots+T_1X+T_0,
\]
whose generic Galois group is $S_d$.
Applying that result to the finitely many conjugacy classes of proper subgroups of $S_d$ gives
\[
\#\{\cc\in E_1:c_d\ne0\}\ll H^{d+1/2+\varepsilon}.
\]
Together with the $O(H^d)$ vectors for which $c_d=0$, this is $O(\theta^{-d}H^{d+1})$ on taking, say, $\varepsilon=1/4$.
The universal discriminant is a nonzero polynomial in $d+1$ variables, so
\[
\#E_2\ll H^d.
\]
Moreover,
\[
\#E_3
\ll\sum_{\theta<m\le H}\left(\frac Hm+1\right)^{d+1}
\ll\theta^{-d}H^{d+1}.
\]

For $E_4$, put $D_0=2d-2$ and $k_0=\lfloor\theta\rfloor+1$.
Since the universal discriminant is homogeneous of degree $D_0$, \cite[Lemma~4.10]{PSW} gives
\[
N_{p,k}:=\#\{\cc\pmod{p^k}:p^k\mid\Delta(g_\cc)\}
\ll p^{k(d+1)-k/D_0}.
\]
When $p^k\le H$, it follows that
\[
\#\{\cc\in S(H):p^k\mid\Delta(g_\cc)\}
\ll H^{d+1}p^{-k/D_0}.
\]
For sufficiently large $H$, we have $p^{k_0}\le H$ for every $p\le\eta$.
Therefore
\[
\#E_4
\ll H^{d+1}\sum_{p\le\eta}p^{-k_0/D_0}
\ll H^{d+1}2^{-k_0/D_0}
\ll\theta^{-d}H^{d+1}.
\]

It remains to bound $E_5$.
Since $S_d$ has a unique subgroup of index $2$, the unique quadratic subfield of $L_\cc$ is $\QQ(\sqrt{\Delta(g_\cc)})$.
Thus, for $\cc\in E_5$,
\[
\Delta(g_\cc)=Dr^2\quad\text{for some }r\in\QQ.
\]
Writing $r=a/b$ in lowest terms gives $b^2\mid |D|$, hence $b=1$.
Consequently $\Delta(g_\cc)=Dm^2$ for some $m\in\ZZ$.
Set
\[
\Phi(T,\cc)=T^2-D\Delta(g_\cc).
\]
This polynomial is absolutely irreducible.
By \cite[Example~1.4]{GKZ}, $\Delta(g_\cc)$ is irreducible in $\CC[\cc]$.
Eisenstein's criterion at $\Delta(g_\cc)$ therefore shows that $\Phi$ is irreducible in $\CC[T,\cc]$.
Give $T$ weight $d-1$ and each coefficient variable weight $1$.
In the notation of \cite[Theorem~1.2]{BCSSV}, there are $n=d+1$ coefficient variables, the cover degree is $2$, $e=d-1$, $F_0=0$, and the weighted degree is $2(d-1)$.
Recall that a weighted homogeneous polynomial is called full in \cite{BCSSV} if the least common multiple of its weights divides its weighted degree.
Here $\operatorname{lcm}(d-1,1,\ldots,1)=d-1$ divides $2(d-1)$, so $\Phi$ is full; its weighted top part is the absolutely irreducible polynomial $\Phi$ itself.
The theorem therefore applies.
For $\Delta(g_\cc)=Dm^2$ we may take $T=Dm$, and $|T|\ll H^{d-1}$ because $|\Delta(g_\cc)|\ll H^{2d-2}$.
The theorem therefore yields
\[
\#E_5
\ll H^{d-1+\sqrt2}(\log H)^{O(1)}
\ll\theta^{-d}H^{d+1}.
\]
Both the implied constant and the exponent in $(\log H)^{O(1)}$ are uniform for fixed $d$: the parameters $n=d+1$, cover degree $2$, weight $e=d-1$, and coefficient height of $\Phi$ in \cite[Theorem~1.2]{BCSSV} are then bounded in terms of $d$, uniformly for $D\in\{-3,-1\}$.
The exponent is $(d+1)-2+2/\sqrt2=d-1+\sqrt2$.
Combining the five bounds proves the lemma.
\end{proof}

\subsection{The large prime Euler product}

Recall from Section~\ref{sec:average} that the multiplicative function $a$ defined in \eqref{eq:def-a} satisfies $a=1*h$, with $h$ as in \eqref{eq:a-conv}.
Thus $h$ is supported on squarefull integers, $h(p)=0$, and $|h(p^r)|\le2$ for $r\ge2$.
We shall use the finite large prime factor $a(m;\eta)$ and its convolution expansion from \eqref{eq:a-eta-conv}; here $P^+$ is defined in \eqref{eq:def-Pplus}.

For a prime $p$ and $r\ge0$, define
\begin{equation}\label{eq:def-xi}
\xi_\cc(p^r)=\frac{\nu_\cc(p^r)}{p^r},
\qquad
\mu_{p,r}(\cc)=\xi_\cc(p^r)-\xi_\cc(p^{r+1}).
\end{equation}
If $\cc\in\mathcal G_{\mathrm{arith}}(H)$, then reduction modulo $p^r$ and Lemma~\ref{lem:nu-bounds}(2) give
\[
\mu_{p,r}(\cc)\ge0,
\qquad
\sum_{r\ge0}\mu_{p,r}(\cc)=1.
\]

\begin{proposition}\label{prop:squarefull-polynomial-values}
Let $A:\NN\to\{\pm1\}$ be multiplicative, with $A(p)=1$ for every prime $p\ge5$ and $A(p^r)=1$ for $p\in\{2,3\}$ and $r\ge1$.
Put
\begin{equation}\label{eq:def-KAp}
h_A=A*\mu,\quad
K_{A,p}(\cc)=1+\sum_{r\ge2}h_A(p^r)\xi_\cc(p^r),
\end{equation}
and let $\cc\in\mathcal G_{\mathrm{arith}}(H)$.
Then, for every $\varepsilon>0$,
\[
\sum_{1\le n\le x}A\bigl(g_\cc(n);\eta\bigr)
=\Bigl(\prod_{5\le p\le\eta}K_{A,p}(\cc)\Bigr)x
+O(x^\varepsilon),
\]
uniformly in $A$ and $\cc$.
\end{proposition}

\begin{proof}
Put
\begin{equation}\label{eq:height-cutoff}
Z=2dHx^d\asymp Hx^d=x^{5d}.
\end{equation}
Omit the at most $d$ integers $n$ for which $g_\cc(n)=0$.
On the remaining terms the convolution formula in Lemma~\ref{lem:squarefull-kernel} gives
\[
A(g_\cc(n);\eta)
=\sum_{\substack{q\mid g_\cc(n)\\P^+(q)\le\eta}}h_A(q).
\]
Including the integer roots temporarily in the congruence counts and then subtracting their contribution gives
\begin{align}\label{eq:trunc-main-err}
\sum_{1\le n\le x}A(g_\cc(n);\eta)
&=x\sum_{\substack{q\ge1\\P^+(q)\le\eta}}
\frac{h_A(q)\nu_\cc(q)}q-\mathcal T_A(Z;\cc)
+O\bigl(\mathcal M_A(Z;\cc)+x^\varepsilon\bigr),
\end{align}
where
\begin{equation}\label{eq:def-trunc-errors}
\begin{aligned}
\mathcal T_A(Z;\cc)
&=x\sum_{\substack{q>Z\\P^+(q)\le\eta}}
\frac{h_A(q)\nu_\cc(q)}q,\quad
\mathcal M_A(Z;\cc)=
\sum_{\substack{q\le Z\\P^+(q)\le\eta}}
|h_A(q)|\nu_\cc(q).
\end{aligned}
\end{equation}
Indeed, before the integer roots are removed, periodicity gives $x\nu_\cc(q)/q+O(\nu_\cc(q))$.
The Euler product $\prod_{p\le\eta}(1+2\lfloor\log Z/\log p\rfloor)$ gives
\[
\sum_{\substack{q\le Z\\P^+(q)\le\eta}}|h_A(q)|
\ll x^\varepsilon,
\]
so the integer roots contribute $O(x^\varepsilon)$.

Since $h_A$ and $\nu_\cc$ are multiplicative, Lemma~\ref{lem:nu-bounds}(2) shows that the local series below are absolutely convergent.
Hence
\begin{equation}\label{eq:trunc-euler-prod}
\sum_{\substack{q\ge1\\P^+(q)\le\eta}}
\frac{h_A(q)\nu_\cc(q)}q
=\prod_{p\le\eta}
\left(1+\sum_{r\ge2}h_A(p^r)\xi_\cc(p^r)\right)
=\prod_{5\le p\le\eta}K_{A,p}(\cc).
\end{equation}
The last equality uses $h_A(p^r)=0$ for $p\in\{2,3\}$.

We next record the uniform bound for $\nu_\cc(p^r)$ used in both error terms.
Fix $p\le\eta$ and put $v=v_p(\cont(g_\cc))$.
Since $\cont(g_\cc)\le\theta$, we have $p^v\le\theta$.
If $r\le v$, then $\nu_\cc(p^r)=p^r\le\theta$.
If $r>v$, write $g_\cc=p^v g_1$ and use Lemma~\ref{lem:nu-bounds}(3), together with
\[
\nu_\cc(p^r)=p^v\nu_{g_1}(p^{r-v}),\quad
\Delta(g_\cc)=p^{v(2d-2)}\Delta(g_1),\quad
v_p(\Delta(g_\cc))\le\theta.
\]
It follows that, for every $r\ge0$,
\begin{equation}\label{eq:nu-smallp-bound}
\nu_\cc(p^r)\ll\theta p^{\theta/2}.
\end{equation}

Let $R_p=\lfloor\log Z/\log p\rfloor$.
The bounds $h_A(p)=0$ and $|h_A(p^r)|\le2$ give
\[
\mathcal M_A(Z;\cc)
\le\prod_{p\le\eta}
\left(1+O\bigl(\theta R_p p^{\theta/2}\bigr)\right).
\]
The logarithm of the product is
\[
\ll\sum_{p\le\eta}\log(\theta R_p)
+\frac{\theta}{2}\sum_{p\le\eta}\log p
\ll\eta+\theta\eta
\ll\frac{\log x}{\log\log\log x}.
\]
Thus, uniformly in $A$ and $\cc$,
\begin{equation}\label{eq:E2-bound}
\mathcal M_A(Z;\cc)\ll x^\varepsilon.
\end{equation}

Finally, Rankin's trick gives
\begin{equation}\label{eq:large-prime-tail-rankin}
|\mathcal T_A(Z;\cc)|
\le xZ^{-1/2}\prod_{p\le\eta}
\left(1+\sum_{r\ge2}
\frac{|h_A(p^r)|\nu_\cc(p^r)}{p^{r/2}}\right).
\end{equation}
By \eqref{eq:nu-smallp-bound}, each inner sum is $O(\theta p^{\theta/2-1})$; hence the product is $O(x^\varepsilon)$ by the same logarithmic estimate.
Since $Z\asymp x^{5d}$, we obtain
\begin{equation}\label{eq:large-prime-tail-bound}
|\mathcal T_A(Z;\cc)|\ll1.
\end{equation}
Substitution in \eqref{eq:trunc-main-err} proves the proposition.
\end{proof}

For the $j=0$ family, put
\begin{equation}\label{eq:kappa-p-def}
\kappa_p(\cc)=1+\sum_{r\ge2}h(p^r)\xi_\cc(p^r).
\end{equation}

The following convexity observation will be used for both families.

\begin{lemma}\label{lem:local-convexity}
Let $\cc\in\mathcal G_{\mathrm{arith}}(H)$ and $p\ge5$.
Suppose that $b(p^r)\in\{\pm1\}$ for $r\ge0$, with $b(1)=b(p)=1$, and put
\[
K_{p,b}(\cc)=\sum_{r\ge0}b(p^r)\mu_{p,r}(\cc).
\]
Then $|K_{p,b}(\cc)|\le1$ and
\begin{equation}\label{eq:local-convexity}
K_{p,b}(\cc)
=1+\sum_{r\ge2}
\bigl(b(p^r)-b(p^{r-1})\bigr)\xi_\cc(p^r).
\end{equation}
\end{lemma}

\begin{proof}
The bound follows because the $\mu_{p,r}(\cc)$ are nonnegative and sum to $1$.
Summation by parts gives \eqref{eq:local-convexity}.
\end{proof}

We now prove that at least one local factor in the $j=0$ family is strictly smaller than $1$ in absolute value.

\begin{lemma}\label{lem:kappa-strict}
For $\cc\in \mathcal G_0(H)$ there exists a prime $p\ge 5$ such that $|\kappa_p(\cc)|<1$.
\end{lemma}

\begin{proof}
Let $f_\cc$ be the primitive homogenization of $g_\cc$.
Its root field, in the notation of \eqref{eq:def-Kh}, is $K_{f_\cc}$, which embeds into the splitting field $L_\cc$.
Equation~\eqref{eq:def-Gfamilies} gives $\QQ(\sqrt{-3})\not\subset L_\cc$, and hence $\QQ(\sqrt{-3})\not\subset K_{f_\cc}$.
Lemma~\ref{lem:chebotarev-projective} supplies an inert prime $p\ge5$ for which $\{f_\cc=0\}$ has a simple projective point.
Choose $p$ away from the leading coefficient and discriminant of $g_\cc$.
The point is then affine, $p\equiv5\pmod6$, and Hensel lifting gives
\[
\nu_\cc(p^r)=\nu_\cc(p)\ge1\quad\text{for every }r\ge1.
\]

By \eqref{eq:h-prime-powers}, \eqref{eq:kappa-p-def}, and Lemma~\ref{lem:local-convexity},
\begin{equation}\label{eq:kappa-weighted}
\kappa_p(\cc)
=K_{p,a}(\cc)
=\sum_{r\ge 0}a(p^r)\mu_{p,r}(\cc).
\end{equation}
Since $p\equiv 5\pmod 6$, from the definition of $\omega_p$ and \eqref{eq:def-a} this gives
\[
a(p^2)=-1,\quad a(p^3)=1.
\]
Hensel lifting gives, for every $r\ge1$,
\[
\mu_{p,r}(\cc)
=\frac{\nu_\cc(p)}{p^r}\left(1-\frac1p\right)>0.
\]
Hence both signs $-1$ and $+1$ occur with positive weight in this weighted average \eqref{eq:kappa-weighted} for $\kappa_p(\cc)$, so $|\kappa_p(\cc)|<1$.
\end{proof}

\begin{proposition}\label{prop:local-product-tail}
Let $\cc\in\mathcal G_{\mathrm{arith}}(H)$.
For every prime $p\ge5$, choose signs $b(p^r)\in\{\pm1\}$ with $b(1)=b(p)=1$, and let $K_{p,b}(\cc)$ be as in Lemma~\ref{lem:local-convexity}.
Then $\prod_{p\ge5}K_{p,b}(\cc)$ converges and
\begin{equation}\label{eq:local-product-trunc}
\prod_{5\le p\le\eta}K_{p,b}(\cc)
=\prod_{p\ge5}K_{p,b}(\cc)+O(\theta^{-1}),
\end{equation}
uniformly in $\cc$ and in the choices of the signs $b(p^r)$.
\end{proposition}

\begin{proof}
Fix $\cc$ and write $\Delta=\Delta(g_\cc)$.
For $p>\eta$, Lemma~\ref{lem:local-convexity} gives
\[
K_{p,b}(\cc)-1
=\sum_{r\ge 2}\bigl(b(p^r)-1\bigr)
\mu_{p,r}(\cc).
\]
Thus, since $|b(p^r)-1|\le2$,
\begin{equation}\label{eq:kappa-delta-bound}
|K_{p,b}(\cc)-1|
\le 2\sum_{r\ge 2}\mu_{p,r}(\cc)
=2\xi_\cc(p^2)
=2\frac{\nu_\cc(p^2)}{p^2}.
\end{equation}

We now estimate the tail product.
For sufficiently large $H$ we have $\eta>\theta\ge\cont(g_\cc)$.
Thus $p\nmid\cont(g_\cc)$ for every $p>\eta$, and consequently $\nu_\cc(p)\le d$.
If moreover $p\nmid \Delta$, then $g_\cc$ has only simple roots modulo $p$, so Hensel lifting gives $\nu_\cc(p^2)=\nu_\cc(p)\le d$, and thus $\nu_\cc(p^2)/p^2\ll p^{-2}$.
If $p\mid \Delta$, reduction modulo $p$ gives $\nu_\cc(p^2)\le p\nu_\cc(p)\le dp$, and hence $\nu_\cc(p^2)/p^2\ll p^{-1}$.
Therefore, from \eqref{eq:kappa-delta-bound},
\begin{equation}\label{eq:sum-kappa-delta}
\sum_{p>\eta}\bigl|K_{p,b}(\cc)-1\bigr|
\ll \sum_{\substack{p>\eta\\ p\mid \Delta}}\frac1p \;+\; \sum_{p>\eta}\frac1{p^2}.
\end{equation}
The second sum is $\ll \eta^{-1}\ll \theta^{-1}$.
For the first sum, split at $\log x$.
By Lemma~\ref{lem:mert}, we have $\sum_{\eta<p\le \log x}p^{-1}=O(\theta^{-1})$, and hence the same bound for any subcollection of primes in this range.
For primes $p>\log x$ dividing $\Delta$, write $p_1,\dots,p_R$ for the distinct such primes.
Since $\Delta(g_\cc)$ is a nonzero polynomial in the coefficients of total degree $2d-2$ and $|\cc|\le H$, we have $|\Delta|\ll H^{2d-2}$, so
\[
(\log x)^R \le p_1\cdots p_R \le |\Delta| \ll H^{2d-2}=x^{O(1)}.
\]
Thus $R\ll \log x/\log\log x$, and therefore
\[
\sum_{\substack{p>\log x\\ p\mid \Delta}}\frac1p \le \frac{R}{\log x}\ll \frac1{\log\log x}\ll \theta^{-1}.
\]
Substituting into \eqref{eq:sum-kappa-delta} gives
\begin{equation}\label{eq:tail-small}
\sum_{p>\eta}\bigl|K_{p,b}(\cc)-1\bigr|\ll \theta^{-1}.
\end{equation}

For $H$ sufficiently large we have $\eta>4d$, so \eqref{eq:kappa-delta-bound} implies $|K_{p,b}(\cc)-1|\le 1/2$ for all $p>\eta$.
In particular these real factors are positive, and $|\log K_{p,b}(\cc)|\ll |K_{p,b}(\cc)-1|$.
Using \eqref{eq:tail-small} we obtain
\[
\log\Bigl(\prod_{p>\eta}K_{p,b}(\cc)\Bigr)
=\sum_{p>\eta}\log K_{p,b}(\cc)=O(\theta^{-1}),
\]
and therefore
\[
\prod_{p>\eta}K_{p,b}(\cc)=1+O(\theta^{-1}).
\]
In particular, the full Euler product converges, and
\[
\prod_{p\ge 5}K_{p,b}(\cc)
=\Bigl(\prod_{5\le p\le \eta}K_{p,b}(\cc)\Bigr)
\Bigl(1+O(\theta^{-1})\Bigr).
\]
Finally, Lemma~\ref{lem:local-convexity} gives $\bigl|\prod_{5\le p\le \eta}K_{p,b}(\cc)\bigr|\le1$, so expanding the previous display yields \eqref{eq:local-product-trunc}.
\end{proof}

\begin{corollary}\label{cor:kappa-tail}
For $\cc\in\mathcal G_{\mathrm{arith}}(H)$, the product $\prod_{p\ge5}\kappa_p(\cc)$ converges and
\begin{equation}\label{eq:kappa-trunc-full}
\prod_{5\le p\le\eta}\kappa_p(\cc)
=\prod_{p\ge5}\kappa_p(\cc)+O(\theta^{-1}).
\end{equation}
\end{corollary}

\begin{proof}
Apply Proposition~\ref{prop:local-product-tail} with $b=a$, using \eqref{eq:h-prime-powers} and \eqref{eq:kappa-p-def}.
\end{proof}

\subsection{Densities at small primes and the truncated average}

Recall the truncated root number $\rw(\cdot;\eta)$ from \eqref{eq:rw-eta-def}.
We also recall the starred-sum convention \eqref{eq:starred-sum-convention}.

For $\cc\in\mathcal G_{\mathrm{arith}}(H)$, put
\begin{equation}\label{eq:beta-c-def}
\beta_\cc(x)=\frac1x\int_1^x\sign(g_\cc(t))\,dt.
\end{equation}

\begin{proposition}\label{prop:cell-average}
Fix positive integers $b_2,b_3$, put
\[
T=2^{b_2}3^{b_3},\quad M_{r,s}=2^{r+b_2}3^{s+b_3},
\]
and let $A:\NN\to\{\pm1\}$ be multiplicative, with $A(p)=1$ for every prime $p\ge5$ and $A(p^r)=1$ for $p\in\{2,3\}$ and $r\ge1$.
Let $K_{A,p}(\cc)$ be as in \eqref{eq:def-KAp}.
Choose signs $\sigma_{r,s,t}\in\{\pm1\}$ for $r,s\ge0$ and $t\in(\ZZ/T\ZZ)^\times$.
Suppose that, for $m\ne0$,
\begin{equation}\label{eq:cell-law}
\begin{aligned}
w(m;\eta)&=\sign(m)\sigma_{r,s,t}A(|m|;\eta),\\
r&=v_2(m),\quad s=v_3(m),\quad
t\equiv\frac{m}{2^r3^s}\pmod T,
\end{aligned}
\end{equation}
where $t$ is evaluated at its least positive representative, and put $w(0;\eta)=0$.
Define
\[
\lambda_{r,s,t}(\cc)
=\frac1{M_{r,s}}\#\{u\pmod{M_{r,s}}:
g_\cc(u)\equiv2^r3^s t\pmod{M_{r,s}}\}
\]
and
\[
\kappa_{\mathrm{sm}}(\cc)
=\sum_{r,s\ge0}\ \sumstar_{t\bmod T}
\sigma_{r,s,t}\lambda_{r,s,t}(\cc).
\]
Then $|\kappa_{\mathrm{sm}}(\cc)|\le1$ and, for every $\varepsilon>0$,
\[
\sum_{1\le n\le x}w(g_\cc(n);\eta)
=\beta_\cc(x)\kappa_{\mathrm{sm}}(\cc)
\Bigl(\prod_{5\le p\le\eta}K_{A,p}(\cc)\Bigr)x
+O_{b_2,b_3}(x^\varepsilon),
\]
uniformly in $A$ and $\cc\in\mathcal G_{\mathrm{arith}}(H)$.
\end{proposition}

\begin{proof}
Let $h_A=A*\mu$ and let $Z$ be as in \eqref{eq:height-cutoff}.
The values $n$ with $g_\cc(n)=0$ contribute nothing.
On the remaining values, expand $A(|g_\cc(n)|;\eta)$ by Lemma~\ref{lem:squarefull-kernel} and interchange the sums:
\[
\sum_{n\le x}w(g_\cc(n);\eta)
=\sum_{\substack{q\le Z\\P^+(q)\le\eta}}h_A(q)S_A(q),
\]
where, for $g_\cc(n)\ne0$, we write
\[
r(n)=v_2(g_\cc(n)),\quad s(n)=v_3(g_\cc(n)),\quad
t(n)\equiv\frac{g_\cc(n)}{2^{r(n)}3^{s(n)}}\pmod T,
\]
and
\[
S_A(q)=\sum_{\substack{n\le x\\q\mid g_\cc(n)\\g_\cc(n)\ne0}}
\sign(g_\cc(n))\sigma_{r(n),s(n),t(n)}.
\]
Only $q$ with $h_A(q)\ne0$ matter, and these are coprime to $6$.

Write
\[
N_{r,s,t}(\cc)
=\#\{u\pmod{M_{r,s}}:
g_\cc(u)\equiv2^r3^s t\pmod{M_{r,s}}\}.
\]
For fixed $(r,s,t)$, the conditions defining $S_A(q)$ occupy $\nu_\cc(q)N_{r,s,t}(\cc)$ classes modulo $qM_{r,s}$.
Partitioning $[1,x]$ at the real roots of $g_\cc$ into intervals of constant sign gives, on each such class,
\[
\sum_{\substack{n\le x\\n\equiv u\pmod{qM_{r,s}}}}
\sign(g_\cc(n))
=\frac{\beta_\cc(x)x}{qM_{r,s}}+O(1).
\]
Moreover, reduction modulo $2^r$ and $3^s$ gives the uniform lift bound
\begin{equation}\label{eq:cell-lift-bound}
N_{r,s,t}(\cc)
\le T\nu_\cc(2^r)\nu_\cc(3^s).
\end{equation}

Only $r\le R_2=\lfloor\log Z/\log2\rfloor$ and $s\le R_3=\lfloor\log Z/\log3\rfloor$ occur.
By \eqref{eq:nu-smallp-bound} and \eqref{eq:cell-lift-bound},
\[
\sum_{0\le r\le R_2}\sum_{0\le s\le R_3}
\sumstar_{t\bmod T}N_{r,s,t}(\cc)
\ll x^\varepsilon.
\]
Completing the main term to all valuation cells costs at most
\[
\frac{\nu_\cc(2^{R_2+1})}{2^{R_2+1}}
+\frac{\nu_\cc(3^{R_3+1})}{3^{R_3+1}}
\ll Z^{-1}x^\varepsilon.
\]
This is the tail of the cell-density coefficient.
Since $x/(qZ)\le1$, its contribution to the main term is $O(\nu_\cc(q)x^\varepsilon)$ for every $\varepsilon>0$.
Consequently,
\begin{equation}\label{eq:generic-cell-sum}
S_A(q)
=\beta_\cc(x)\kappa_{\mathrm{sm}}(\cc)
\frac{x\nu_\cc(q)}q
+O_{b_2,b_3}\bigl(\nu_\cc(q)x^\varepsilon\bigr).
\end{equation}

The $\lambda_{r,s,t}(\cc)$ are the nonnegative densities of disjoint valuation cells.
The cells with $r<R$ and $s<S$ omit only values for which $2^R\mid g_\cc(n)$ or $3^S\mid g_\cc(n)$, a set of upper density at most
\[
\frac{\nu_\cc(2^R)}{2^R}+\frac{\nu_\cc(3^S)}{3^S}.
\]
This tends to $0$ by Lemma~\ref{lem:nu-bounds}(2), so the cell densities sum to $1$.
It follows that $|\kappa_{\mathrm{sm}}(\cc)|\le1$.

Apply \eqref{eq:generic-cell-sum} with $\varepsilon/2$ in place of $\varepsilon$ and insert it in the divisor expansion.
With $\mathcal T_A$ and $\mathcal M_A$ as defined in \eqref{eq:def-trunc-errors}, the tail of the main $q$-sum is $\mathcal T_A(Z;\cc)$, while the accumulated error is bounded by $x^{\varepsilon/2}\mathcal M_A(Z;\cc)$.
Equations \eqref{eq:large-prime-tail-bound} and \eqref{eq:E2-bound}, the latter also with $\varepsilon/2$ in place of $\varepsilon$, therefore give an error $O(x^\varepsilon)$.
Finally, \eqref{eq:trunc-euler-prod} supplies the stated large prime product.
\end{proof}

For the $j=0$ family, define
\begin{equation}\label{eq:kappa23-def}
\kappa_{2,3}(\cc)
=\sum_{r,s\ge0}\ \sumstar_{t\bmod36}
\bigl(-\omega_2(2^r3^s t)\omega_3(2^r3^s t)\bigr)
\lambda_{r,s,t}(\cc),
\end{equation}
where
\[
\lambda_{r,s,t}(\cc)
=\frac1{2^{r+2}3^{s+2}}
\#\{u\pmod{2^{r+2}3^{s+2}}:
g_\cc(u)\equiv2^r3^s t\pmod{2^{r+2}3^{s+2}}\}.
\]

\begin{proposition}\label{prop:asymp}
For every $\varepsilon>0$ and every $\cc\in\mathcal G_{\mathrm{arith}}(H)$,
\begin{equation}\label{eq:rw-trunc-poly}
\sum_{1\le n\le x}\rw(g_\cc(n);\eta)
=\beta_\cc(x)\kappa_{2,3}(\cc)
\Bigl(\prod_{5\le p\le\eta}\kappa_p(\cc)\Bigr)x
+O(x^\varepsilon).
\end{equation}
Moreover $|\kappa_{2,3}(\cc)|\le1$.
\end{proposition}

\begin{proof}
Apply Proposition~\ref{prop:cell-average} with $A=a$, $(b_2,b_3,T)=(2,2,36)$, and
\[
\sigma_{r,s,t}
=-\omega_2(2^r3^s t)\omega_3(2^r3^s t).
\]
The cell law \eqref{eq:cell-law} is the factorisation of $\rw(\cdot;\eta)$ verified in the proof of Proposition~\ref{prop:rw-interval}: here $a(|m|;\eta)=\prod_{5\le p\le\eta}\omega_p(m)$, and the negative cells supply the factor $\sign(m)$.
Moreover, $K_{A,p}(\cc)=\kappa_p(\cc)$ and $\kappa_{\mathrm{sm}}(\cc)=\kappa_{2,3}(\cc)$.
\end{proof}

\begin{lemma}\label{lem:final-assembly}
Let $\mathcal G\subset\mathcal G_{\mathrm{arith}}(H)$ satisfy $\#(S(H)\setminus\mathcal G)\ll\theta^{-d}H^{d+1}$, and let $\mathcal E\subset S(H)$ have size $O(\theta^{-1/3}H^{d+1})$.
Suppose that uniformly bounded functions $w,w_\eta:\ZZ\to\CC$ and real numbers $K_{2,3}(\cc),K_p(\cc)$ satisfy, uniformly for $\cc\in\mathcal G$,
\begin{align*}
\sum_{n\le x}w_\eta(g_\cc(n))
&=\beta_\cc(x)K_{2,3}(\cc)
  \Bigl(\prod_{5\le p\le\eta}K_p(\cc)\Bigr)x+O(x^{1/2}),\\
\prod_{5\le p\le\eta}K_p(\cc)
&=\prod_{p\ge5}K_p(\cc)+O(\theta^{-1}),
\end{align*}
where the full Euler product converges for every $\cc\in\mathcal G$.
Suppose also that, for $\cc\in S(H)\setminus\mathcal E$,
\[
\left|\sum_{n\le x}
\bigl(w(g_\cc(n))-w_\eta(g_\cc(n))\bigr)\right|
\ll\theta^{-1/3}x.
\]
Assume, for every $\cc\in\mathcal G$, that
\[
|K_{2,3}(\cc)|\le1,\quad |K_p(\cc)|\le1\quad\text{for every }p\ge5,
\]
and that for each such $\cc$ there is a prime $p\ge5$ for which $|K_p(\cc)|<1$.
Then, outside a subset of $S(H)$ of size $O(\theta^{-1/3}H^{d+1})$,
\[
\sum_{n\le x}w(g_\cc(n))
=K_\cc x+O(\theta^{-1/3}x),
\quad
K_\cc=\sign(c_d)K_{2,3}(\cc)\prod_{p\ge5}K_p(\cc),
\]
and $|K_\cc|<1$.
\end{lemma}

\begin{proof}
Exclude $\mathcal E$, $S(H)\setminus\mathcal G$, and the vectors with $|c_d|<Hx^{-1/2}$.
The last set has size $O(H^{d+1}x^{-1/2})=o(\theta^{-1/3}H^{d+1})$.
For every remaining vector, $|c_i|\le H$ for $i<d$ and
$|c_d|\ge Hx^{-1/2}$.
Thus every complex root $\alpha$ of $g_\cc$ satisfies
\[
|\alpha|\le1+\max_{i<d}|c_i/c_d|\le1+x^{1/2}
\]
by comparing the leading term with the lower-degree terms.
Hence
\begin{equation}\label{eq:beta-leading-sign}
\beta_\cc(x)=\sign(c_d)+O(x^{-1/2}).
\end{equation}
Insert this and the Euler product tail into the truncated asymptotic.
Since $x^{1/2}=o(\theta^{-1}x)$, the result is $K_\cc x+O(\theta^{-1}x)$.
The truncation error estimate gives the claimed final error.
The bounds on the local factors, with one strict inequality, give $|K_\cc|<1$.
\end{proof}

\subsection{Averages for almost all polynomials: \texorpdfstring{$j=0$}{j=0}}

\begin{proof}[Proof of \eqref{eq:main-fixed} in Theorem~\ref{thm:main}]
Let $\mathcal E$ be the exceptional set supplied by Corollary~\ref{cor:pointwise-trunc}.
Put
\begin{equation}\label{eq:def-B0}
\mathcal B_0(H)=\mathcal E\cup\bigl(S(H)\setminus\mathcal G_0(H)\bigr)
\cup\{\cc\in S(H):|c_d|<Hx^{-1/2}\}.
\end{equation}
Apply Lemma~\ref{lem:final-assembly} with
\[
w=\rw,\quad w_\eta=\rw(\,\cdot\,;\eta),\quad
K_{2,3}=\kappa_{2,3},\quad K_p=\kappa_p.
\]
Its hypotheses are Corollary~\ref{cor:pointwise-trunc}, Lemma~\ref{lem:generic-families}, Proposition~\ref{prop:asymp}, Corollary~\ref{cor:kappa-tail}, Lemma~\ref{lem:local-convexity}, and Lemma~\ref{lem:kappa-strict}.
Thus
\begin{equation}\label{eq:kappa-c-product}
\kappa_\cc=\sign(c_d)\kappa_{2,3}(\cc)\prod_{p\ge5}\kappa_p(\cc)
\end{equation}
satisfies $|\kappa_\cc|<1$, and \eqref{eq:main-fixed} holds for every $\cc\notin\mathcal B_0(H)$.
Moreover,
\[
\#\mathcal B_0(H)\ll\theta^{-1/3}H^{d+1}.
\]
\end{proof}

\section{The \texorpdfstring{$j=1728$}{j=1728} family}\label{sec:j1728-generic}

The arguments parallel those in Sections~\ref{sec:fixed-forms} and~\ref{sec:poly-sum}.
We use the local root numbers from Lemma~\ref{lem:W23-1728}, the factors defined in Definition~\ref{def:omega23-1728} and \eqref{eq:def-omega-p-star}, and the factorisation in Proposition~\ref{prop:rootnumber-factor-1728}.
The following table summarises the systematic substitutions; below we prove only the changes forced by the local formulas.
\begin{center}
\begin{tabular}{|l|c|c|}
\hline
 & \emph{$j=0$} & \emph{$j=1728$} \\
\hline
valuation period at $p\ge5$ & $6$ & $4$ \\
\hline
factors in $P$ & $3\nmid e_i$ & $e_i$ odd \\
\hline
factors in $Q$ & $e_i\equiv2,4\pmod6$ & $e_i\equiv2\pmod4$ \\
\hline
factors in $R$ & $6\nmid e_i$ & $4\nmid e_i$ \\
\hline
correcting character & $\bigl(\frac{-3}{p}\bigr)$ & $\bigl(\frac{-1}{p}\bigr)$ \\
\hline
auxiliary field & $\QQ(\sqrt{-3})$ & $\QQ(i)$ \\
\hline
small prime cell data & $(2,2,36)$ & $(4,1,48)$\\
\hline
\end{tabular}
\end{center}
The last row records $(b_2,b_3,T)$ in the cell notation of Proposition~\ref{prop:cell-average}.
Unstarred notation refers to the first family; a superscript $\star$ denotes its counterpart in the second family.

\subsection{Fixed forms}

The proof of Theorem~\ref{thm:asy}(ii) follows the same strategy as the proof of Theorem~\ref{thm:asy}(i), but the real sign has to be separated before the product of the finite local factors is periodic.
Recall $S_{\mathrm{prim}}(x)$ and $K_h$ from \eqref{eq:def-Sprim} and \eqref{eq:def-Kh}.

\begin{proof}[Proof of Theorem~\ref{thm:asy}(ii)]
Write
\[
G_{\mathrm{red}}=\prod_{i=1}^r g_i,\quad
Q_G=\prod_{e_i\equiv2\pmod4}g_i
\]
and put
\[
P_G=\prod_{e_i\text{ odd}}g_i,
\quad
R_G=\prod_{4\nmid e_i}g_i,
\]
with the convention that an empty product is $1$.
The form $P_G$ controls the large square exceptional set, and $R_G$ records the real sign.

Let $\Sigma$ be the finite set of primes $p$ such that either $p\mid 6c$, or $\overline{G_{\mathrm{red}}}$ is not squarefree in $\FF_p[x,y]$.
For $p\notin \Sigma$ and $\gcd(m,n)=1$, at most one of the values $g_i(m,n)$ is divisible by $p$.
For $p\ge 5$ put
\[
u_p^\star(m,n)
=\omega_p^\star(G(m,n))
\left(\frac{-1}{p}\right)^{-v_p(Q_G(m,n))}.
\]
Suppose that $p\notin\Sigma$ divides $G_{\mathrm{red}}(m,n)$, and let $g_j$ be the unique factor divisible by $p$.
Write $k=v_p(g_j(m,n))$.
If $e_j$ is odd, then $p^2\nmid P_G(m,n)$ forces $k=1$, so $v_p(G(m,n))\equiv1$ or $3\pmod4$ and $u_p^\star(m,n)=1$.
The same conclusion is immediate when $4\mid e_j$.
Finally, if $e_j\equiv2\pmod4$, then
\[
v_p(G(m,n))\equiv2k\pmod4,\quad v_p(Q_G(m,n))=k,
\]
so
\[
\omega_p^\star(G(m,n))
=\left(\frac{-1}{p}\right)^k,
\]
which is cancelled by the correcting factor in $u_p^\star$.
Thus $u_p^\star(m,n)=1$ whenever $p^2\nmid P_G(m,n)$; this is also immediate when $p\nmid G_{\mathrm{red}}(m,n)$.

For $z\neq 0$ define $z_2=z/2^{v_2(z)}$ and
\[
J_4(z)=\sign(z)\left(\frac{-1}{|z_2|}\right),
\quad
J_{-8}(z)=\sign(z)\left(\frac{-2}{|z_2|}\right).
\]
The functions $J_4$ and $J_{-8}$ depend only on the signed odd unit $z_2$ modulo $4$ and $8$, respectively.
Since $\bigl(\frac{-1}{3}\bigr)=-1$, we have
\begin{equation}\label{eq:asy-chi-star}
\prod_{p\ge 5}\left(\frac{-1}{p}\right)^{v_p(z)}
=(-1)^{v_3(z)}\sign(z)J_4(z)
\end{equation}
for every $z\neq0$.

For primitive $(m,n)$ with $G(m,n)\neq 0$, define
\begin{align*}
b_G(m,n)
&=-\sign(c)\,\rw_2^\star(G(m,n))J_{-8}(G(m,n))\,\omega_3^\star(G(m,n)) \\
&\quad \times (-1)^{v_3(Q_G(m,n))}J_4(Q_G(m,n)) \\
&\quad \times \prod_{\substack{q\in \Sigma\\ q\ge 5}}
\omega_q^\star(G(m,n))
\left(\frac{-1}{q}\right)^{-v_q(Q_G(m,n))}.
\end{align*}
Since $\omega_p^\star(t)=1$ unless $v_p(t)\equiv2\pmod4$, extending the product in Proposition~\ref{prop:rootnumber-factor-1728} to all primes $p\ge 5$ does not change it.
Multiplying and dividing by $\prod_{p\ge 5}\bigl(\frac{-1}{p}\bigr)^{v_p(Q_G(m,n))}$, and using \eqref{eq:asy-chi-star}, we obtain
\begin{equation}\label{eq:asy-factor-main-star}
\rw^\star(G(m,n))
=\sign(R_G(m,n))b_G(m,n)
\prod_{\substack{p\notin \Sigma\\ p\ge 5}}u_p^\star(m,n).
\end{equation}
Indeed, $\omega_2^\star(G)=\rw_2^\star(G)\sign(G)J_{-8}(G)$, and the exponent of $g_i$ in $G\cdot Q_G$ is odd precisely when $4\nmid e_i$.
Hence
\[
\sign(G(m,n))\sign(Q_G(m,n))=\sign(c)\sign(R_G(m,n)).
\]

We verify the hypotheses of Lemma~\ref{lem:periodic-approximation} with $D=G_{\mathrm{red}}$, $P=P_G$, $R=R_G$, $b=b_G$, and $u_p=u_p^\star$.
Given integers $H_q\ge1$ for $q\in\Sigma$, consider a pair satisfying $q^{H_q}\nmid G_{\mathrm{red}}(m,n)$ for every $q\in\Sigma$.
Then $v_q(g_i(m,n))\le H_q-1$ for all $i$, and therefore
\begin{align*}
v_q(G(m,n))&\le L_q:=v_q(c)+\sum_i e_i(H_q-1),\\
v_q(Q_G(m,n))&\le L_q':=\sum_{e_i\equiv2\pmod4}(H_q-1).
\end{align*}
Set
\[
B_2=\max(L_2+4,L_2'+2),\quad
B_3=1+\max(L_3,L_3'),
\]
and, for $q\in \Sigma$ with $q\ge 5$,
\[
B_q=1+\max(L_q,L_q').
\]
If primitive pairs $(m,n)$ and $(m',n')$ are congruent modulo
\[
M_H=2^{B_2}3^{B_3}
\prod_{\substack{q\in \Sigma\\q\ge5}}q^{B_q}
\]
and satisfy the preceding nondivisibility conditions, then \eqref{eq:asy-valuations} shows that $b_G$ is constant on their residue class: at $2$ one needs the signed unit of $G$ modulo $16$ and that of $Q_G$ modulo $4$; at $3$ only the relevant valuations occur; and at $q\ge5$ the factors depend on the two valuations.
Extend $b_G$ arbitrarily across the excluded classes.
Likewise, if $p\notin\Sigma$ and $p^K\nmid P_G(m,n)$, the valuation below $K$ of the unique divisible odd-exponent factor is determined modulo $p^K$; the even-exponent cases give $u_p^\star=1$ by the preceding calculation.
Hence the hypotheses of Lemma~\ref{lem:periodic-approximation} hold, giving the required limit $C_G^\star$ and the bound $|C_G^\star|\le1$.
Since $G$ is nonzero, the pairs in $S_{\mathrm{prim}}(N)$ for which $G(m,n)=0$ number $O_G(N)$.
Away from this density-zero set, for each $\varepsilon\in\{\pm1\}$,
\[
\mathbf 1_{\{\rw^\star(G(m,n))=\varepsilon\}}
=\frac{1+\varepsilon\rw^\star(G(m,n))}{2}.
\]
Taking averages proves the asserted density formula.

Assume now that there exists $j_0$ with $e=e_{j_0}\equiv1$ or $3\pmod4$ and
\[
\QQ(i)\not\subset K_{g_{j_0}}.
\]
Write $g_0=g_{j_0}$.
By Lemma~\ref{lem:chebotarev-projective}, there are infinitely many primes $q\equiv3\pmod4$ for which $\{g_0=0\}$ has a simple point in $\PP^1(\FF_q)$.
Choose one such prime with $q\notin\Sigma$, $q\ge Y_0(P_G)$, and $q>\deg G_{\mathrm{red}}$.
By Lemma~\ref{lem:exact-valuation-classes}, there are two primitive classes modulo $q^3$: on the first $q\nmid G_{\mathrm{red}}(m,n)$, while on the second
\[
v_q(g_0(m,n))=2,\quad q\nmid g_j(m,n)\quad\text{for }j\ne j_0.
\]

For each $\ell\in\Sigma$, choose a primitive $\mathbf a_\ell\in\ZZ_\ell^2$ with $G_{\mathrm{red}}(\mathbf a_\ell)\ne0$, and then choose $B_\ell$ large enough that every factor in $b_G$ is fixed on $\mathbf a_\ell\pmod{\ell^{B_\ell}}$.
Put
\[
M_0=\prod_{\ell\in\Sigma}\ell^{B_\ell}.
\]
The Chinese remainder theorem combines these common local conditions with the two $q$-adic classes into two primitive classes modulo $M=M_0q^3$.
Choose a nonempty open polygonal region $\mathcal R\subset(0,1]^2$ on which $\sign(R_G)$ is constant.
Every prime at which $\{P_G=0\}$ is not squarefree belongs to $\Sigma$, because $P_G$ is a subproduct of $G_{\mathrm{red}}$; hence it divides $M$.

Lemma~\ref{lem:squarefree-class-density}, applied to $P_G$, the sector $\mathcal R$, and the two classes modulo $M$, gives sets of positive lower density $\mathcal A_N,\mathcal B_N\subset S_{\mathrm{prim}}(N)$ in the chosen sector, belonging to the two respective congruence classes and satisfying $p^2\nmid P_G(m,n)$ for every $p\nmid M$.
Removing the $O_G(N)$ points on which $G(m,n)=0$ does not change these positive lower densities.

For $p\nmid M$, good reduction and the condition $p^2\nmid P_G(m,n)$ give $u_p^\star(m,n)=1$.
The local data at every prime dividing $M$ other than $q$ agree by construction, and $\sign(R_G)$ is constant on $\mathcal R$.
On $\mathcal A_N$, we have $q\nmid G_{\mathrm{red}}(m,n)$, so $u_q^\star(m,n)=1$.
On $\mathcal B_N$, we have $v_q(g_0(m,n))=2$.
Since $e$ is odd, $v_q(G(m,n))\equiv2\pmod4$ and $v_q(Q_G(m,n))=0$, so
\[
u_q^\star(m,n)=\omega_q^\star(G(m,n))
=\left(\frac{-1}{q}\right)=-1.
\]
Thus the two constructed sets of positive lower density have constant and opposite root numbers.
Since $\#S_{\mathrm{prim}}(N)\sim(6/\pi^2)N^2$, both signs occur with positive relative density inside $S_{\mathrm{prim}}(N)$, and the existing average satisfies $|C_G^\star|<1$.
\end{proof}

\subsection{Averages for almost all polynomials: \texorpdfstring{$j=1728$}{j=1728}}

Let $d\ge2$, take $H$ sufficiently large in terms of $d$, and put
\[
x=H^{1/(4d)},\quad
\eta=\frac{\log x}{\log\log x},\quad
\theta=\frac{\log\log x}{\log\log\log x}.
\]
We recall $S(H)$ and $g_\cc$ from \eqref{eq:def-SH} and \eqref{eq:def-gc}, and $\nu_\cc$ from \eqref{eq:def-nu}.
The root number $\rw^\star(n)$ and the factors $\omega_p^\star(n)$ are as in Proposition~\ref{prop:rootnumber-factor-1728}, Definition~\ref{def:omega23-1728}, and \eqref{eq:def-omega-p-star}.
We also recall the starred-sum convention \eqref{eq:starred-sum-convention}.

For $n\in\ZZ\setminus\{0\}$ define the function
\begin{equation}\label{eq:def-a-star}
a^\star(n)=\prod_{\substack{p^2\mid n\\ p\ge 5}}\omega_p^\star(n).
\end{equation}
Its restriction to $\NN$ is multiplicative and satisfies $a^\star(-n)=a^\star(n)$.
Let $h^\star=a^\star*\mu$ on $\NN$, so $h^\star$ is multiplicative, satisfies $h^\star(p)=0$ by \eqref{eq:def-omega-p-star}, and is supported on squarefull integers.
Moreover, $h^\star(d)=0$ whenever $\gcd(d,6)>1$.
For $z\in[5,\infty]$ and $n\ne0$, define
\[
a^\star(n;z)=\sum_{\substack{d\mid n\\P^+(d)\le z}}h^\star(d),\quad
c^\star(u,q;z)=\sum_{\substack{d\ge1,\ P^+(d)\le z\\\gcd(d,q)\mid u}}
h^\star(d)\frac{\gcd(d,q)}d,
\]
where the restrictions involving $P^+(d)$ are omitted when $z=\infty$; also set $a^\star(0;z)=0$.
These are the truncations associated with $a^\star$ in Lemma~\ref{lem:squarefull-kernel}, and we write $c^\star(u,q)=c^\star(u,q;\infty)$.

For $r,s\ge 0$ set
\[
M^\star_{r,s}=2^{r+4}3^{s+1},
\quad
Q^\star_{r,s}(q)=\operatorname{lcm}(q,M^\star_{r,s}).
\]
For $q\ge 1$, $u\pmod q$ and $t\bmod48$ with $\gcd(t,6)=1$ define
\[
\lambda^\star_{r,s,t}(u,q)=
\begin{cases}
\dfrac{q}{Q^\star_{r,s}(q)} & \text{if } u\equiv 2^{r}3^{s}t
\pmod{\gcd(q,M^\star_{r,s})},\\
0 & \text{otherwise}.
\end{cases}
\]
We then set
\begin{equation}\label{eq:def-rho-star}
\rho^\star(u,q)=\sum_{r,s\ge 0}\ \sumstar_{t\bmod48}
\bigl(-\omega_2^\star(2^{r}3^{s}t)\,\omega_3^\star(2^{r}3^{s}t)\bigr)\,
\lambda^\star_{r,s,t}(u,q),
\end{equation}
whose defining series converges absolutely.

For $n\ne0$ and $z\in[5,\infty]$, define
\begin{equation}\label{eq:rw-star-z}
\rw^\star(n;z)
=-\omega_2^\star(n)\omega_3^\star(n)a^\star(n;z),
\quad \rw^\star(0;z)=0.
\end{equation}
Thus $\rw^\star(n;\infty)=\rw^\star(n)$, while for finite $z$ this is the root number product truncated at $p\le z$.

\begin{proposition}\label{prop:rw-interval-star}
Let $q\ge1$, let $u\pmod q$, and let $I\subset\RR$ be a bounded interval.
With $I_+,I_-$ and $\langle I\rangle$ as in \eqref{eq:interval-notation}, for $z\in[5,\infty]$ and every $\varepsilon>0$,
\[
\sum_{\substack{n\in I\\ n\equiv u\pmod q}}\rw^\star(n;z)
=\rho^\star(u,q)c^\star(u,q;z)
\frac{|I_+|-|I_-|}{q}
+O\bigl(\langle I\rangle^{1/2+\varepsilon}\bigr),
\]
uniformly in $q,u,I,z$.
Moreover $|\rho^\star(u,q)|\le1$, and $\rho^\star(u,q)$ depends on $u$ only modulo $2^{v_2(q)}3^{v_3(q)}$.
\end{proposition}

\begin{proof}
Apply Proposition~\ref{prop:cell-progressions} with $A=a^\star$, $(b_2,b_3)=(4,1)$, and
\[
\sigma_{r,s,t}
=-\omega_2^\star(2^r3^s t)\omega_3^\star(2^r3^s t).
\]
For positive $n$, Proposition~\ref{prop:rootnumber-factor-1728} gives the required cell law: the cell fixes the odd part modulo $16$ and the $3$-adic valuation modulo $4$.
For negative $n$, the same cell fixes the local root numbers; with $n_2$ denoting the odd part from \eqref{eq:def-odd-part-star},
\[
\left(\frac{-2}{|n_2|}\right)
=-\left(\frac{-2}{3^s t}\right),
\]
which supplies precisely the factor $\sign(n)$.
Proposition~\ref{prop:cell-progressions} now gives all the stated assertions.
\end{proof}

\begin{lemma}\label{lem:j1728-trunc-AP}
For each fixed $d\ge2$ and every fixed $0<\varepsilon<1/18$, the function $n\mapsto\rw^\star(n)-\rw^\star(n;\eta)$ satisfies hypotheses \textup{(B)} and \textup{(AP)} of Theorem~\ref{thm:transference}.
\end{lemma}

\begin{proof}
Subtract the $z=\eta$ case of Proposition~\ref{prop:rw-interval-star} from its $z=\infty$ case.
On an admissible progression, Lemma~\ref{lem:squarefull-tail}, applied to $A=a^\star$, gives
\[
|c^\star(u,q)-c^\star(u,q;\eta)|\ll\theta^{-1}.
\]
The error term calculation in Lemma~\ref{lem:j0-trunc-AP} applies without change and proves \textup{(AP)}.
Hypothesis \textup{(B)} is immediate.
\end{proof}

\begin{corollary}\label{cor:pointwise-trunc-star}
There exists a subset $\mathcal E^\star \subset S(H)$ with
\[
\#\mathcal E^\star \ll \theta^{-\frac{1}{3}}\,H^{d+1},
\]
such that for every $\cc\in S(H)\setminus\mathcal E^\star$, we have
\[
\left|
\sum_{1\le n\le x}
\Bigl(\rw^\star\bigl(g_{\cc}(n)\bigr)-\rw^\star\bigl(g_{\cc}(n); \eta\bigr)\Bigr)
\right|
\ll \theta^{-\frac{1}{3}}\,x.
\]
\end{corollary}
\begin{proof}
This follows from the proof of Corollary~\ref{cor:pointwise-trunc}, using the starred progression estimate in Lemma~\ref{lem:j1728-trunc-AP}.
\end{proof}

Lemma~\ref{lem:generic-families}, applied to the set $\mathcal G_{1728}(H)$ defined in \eqref{eq:def-Gfamilies}, gives
\[
\#\bigl(S(H)\setminus\mathcal G_{1728}(H)\bigr)
\ll\theta^{-d}H^{d+1},
\]
and this contribution is included in the final exceptional set.

For the large prime factors, recall $h^\star=a^\star*\mu$ and $\xi_\cc(p^r)=\nu_\cc(p^r)/p^r$ from \eqref{eq:def-xi}, and put
\begin{equation}\label{eq:kappa-p-def-star}
\kappa_p^\star(\cc)=1+\sum_{r\ge 2}h^\star(p^r)\,\xi_\cc(p^r).
\end{equation}

We next show that one local factor is strictly smaller than $1$ in absolute value.

\begin{lemma}\label{lem:kappa-strict-star}
For $\cc\in \mathcal G_{1728}(H)$ there exists a prime $p\ge 5$ such that $|\kappa_p^\star(\cc)|<1$.
\end{lemma}

\begin{proof}
Let $f_\cc$ be the primitive homogenization of $g_\cc$.
Its root field, in the notation of \eqref{eq:def-Kh}, is $K_{f_\cc}$, which embeds into the splitting field $L_\cc$.
Equation~\eqref{eq:def-Gfamilies} gives $\QQ(i)\not\subset L_\cc$, and hence $\QQ(i)\not\subset K_{f_\cc}$.
Lemma~\ref{lem:chebotarev-projective} supplies an inert prime $p\ge5$ for which $\{f_\cc=0\}$ has a simple projective point.
Choose $p$ away from the leading coefficient and discriminant of $g_\cc$.
The point is then affine, $p\equiv3\pmod4$, and
\[
\nu_\cc(p^r)=\nu_\cc(p)\ge1\quad\text{for every }r\ge1.
\]

By \eqref{eq:kappa-p-def-star}, the identity
$h^\star(p^r)=a^\star(p^r)-a^\star(p^{r-1})$, and
Lemma~\ref{lem:local-convexity} applied with $b=a^\star$, we have
\[
\kappa_p^\star(\cc)
=K_{p,a^\star}(\cc)
=\sum_{r\ge 0}a^\star(p^r)\mu_{p,r}(\cc).
\]
Since $p\equiv 3\pmod 4$, from \eqref{eq:def-omega-p-star} we have
\[
a^\star(p^2)=\left(\frac{-1}{p}\right)=-1,\quad
a^\star(p^4)=1.
\]
Moreover, Hensel lifting gives
\[
\mu_{p,2}(\cc)=\frac{\nu_\cc(p)}{p^2}\left(1-\frac1p\right)>0,\quad
\mu_{p,4}(\cc)=\frac{\nu_\cc(p)}{p^4}\left(1-\frac1p\right)>0.
\]
Hence in the weighted average for $\kappa_p^\star(\cc)$, both signs $-1$ and $+1$ occur with positive weight.
Therefore $|\kappa_p^\star(\cc)|<1$.
\end{proof}

\begin{proposition}\label{prop:kappa-tail-star}
Let $\cc\in \mathcal G_{\mathrm{arith}}(H)$.
The Euler product $\prod_{p\ge 5}\kappa_p^\star(\cc)$ converges, and
\[
\prod_{5\le p\le \eta}\kappa_p^\star(\cc)
=\prod_{p\ge 5}\kappa_p^\star(\cc)+O(\theta^{-1}).
\]
\end{proposition}

\begin{proof}
Apply Proposition~\ref{prop:local-product-tail} with $b=a^\star$.
The identity $K_{p,b}(\cc)=\kappa_p^\star(\cc)$ follows from $h^\star(p^r)=a^\star(p^r)-a^\star(p^{r-1})$ and \eqref{eq:kappa-p-def-star}.
\end{proof}

The contribution for small primes is as follows.
\begin{proposition}\label{prop:asymp-star}
For every $\varepsilon>0$ and every $\cc\in \mathcal G_{\mathrm{arith}}(H)$,
\begin{equation*}
\sum_{1\le n\le x}\rw^\star\bigl(g_\cc(n); \eta\bigr)
=\beta_\cc(x)\kappa_{2,3}^\star(\cc)
\Bigg(\prod_{5\le p\le \eta}\kappa_p^\star(\cc)\Bigg)\,x
+O\bigl(x^{\varepsilon}\bigr),
\end{equation*}
where $\beta_\cc(x)$ is defined in \eqref{eq:beta-c-def}, $\kappa_p^\star(\cc)$ is as in \eqref{eq:kappa-p-def-star}, and
\begin{equation}\label{eq:kappa23-def-star}
\kappa_{2,3}^\star(\cc)
=\sum_{r,s\ge 0}\ \sumstar_{t\bmod48}
\bigl(-\omega_2^\star(2^{r}3^{s}t)\,\omega_3^\star(2^{r}3^{s}t)\bigr)\,
\lambda_{r,s,t}^\star(\cc),
\end{equation}
with
\[
\lambda_{r,s,t}^\star(\cc)
=\frac{1}{M_{r,s}^\star}\,
\#\Bigl\{u\bmod M_{r,s}^\star:
g_\cc(u)\equiv2^r3^s t\pmod{M_{r,s}^\star}\Bigr\},
\]
and $M_{r,s}^\star=2^{r+4}3^{s+1}$.
In particular, $|\kappa_{2,3}^\star(\cc)|\le 1$.
\end{proposition}
\begin{proof}
Take $A=a^\star$, $b_2=4$, $b_3=1$, and
\[
\sigma_{r,s,t}
=-\omega_2^\star(2^r3^s t)\omega_3^\star(2^r3^s t).
\]
For positive $m$, the factorisation of $\rw^\star(m;\eta)$ gives the cell law \eqref{eq:cell-law}; for negative $m$, the calculation in the proof of Proposition~\ref{prop:rw-interval-star} supplies the extra factor $\sign(m)$.
The densities in Proposition~\ref{prop:cell-average} are then exactly $\lambda_{r,s,t}^\star(\cc)$, its small prime constant is $\kappa_{2,3}^\star(\cc)$, and $K_{A,p}(\cc)=\kappa_p^\star(\cc)$.
This proves the stated formula and the bound $|\kappa_{2,3}^\star(\cc)|\le1$.
\end{proof}

We now prove \eqref{eq:main-fixed-star} in Theorem~\ref{thm:main}.
\begin{proof}[Proof of \eqref{eq:main-fixed-star} in Theorem~\ref{thm:main}]
Let $\mathcal E^\star$ be the exceptional set in Corollary~\ref{cor:pointwise-trunc-star}, and put
\begin{equation}\label{eq:def-B1728}
\mathcal B_{1728}(H)
=\mathcal E^\star\cup\bigl(S(H)\setminus\mathcal G_{1728}(H)\bigr)
\cup\{\cc\in S(H):|c_d|<Hx^{-1/2}\}.
\end{equation}
Apply Lemma~\ref{lem:final-assembly} with
\[
w=\rw^\star,\quad
w_\eta=\rw^\star(\,\cdot\,;\eta),\quad
K_{2,3}=\kappa_{2,3}^\star,\quad K_p=\kappa_p^\star.
\]
The truncation and genericity hypotheses are Corollary~\ref{cor:pointwise-trunc-star} and Lemma~\ref{lem:generic-families}.
The remaining inputs are Propositions~\ref{prop:asymp-star} and~\ref{prop:kappa-tail-star}, together with Lemmas~\ref{lem:local-convexity} and~\ref{lem:kappa-strict-star}.
Consequently,
\begin{equation}\label{eq:kappa-c-product-star}
\kappa_\cc^\star
=\sign(c_d)\kappa_{2,3}^\star(\cc)\prod_{p\ge5}\kappa_p^\star(\cc)
\end{equation}
satisfies $|\kappa_\cc^\star|<1$, and \eqref{eq:main-fixed-star} holds for every $\cc\notin\mathcal B_{1728}(H)$.
Also,
\[
\#\mathcal B_{1728}(H)\ll\theta^{-1/3}H^{d+1}.
\]

Finally, set
\[
\mathcal E_d(H)=\mathcal B_0(H)\cup\mathcal B_{1728}(H).
\]
This set has the size asserted in Theorem~\ref{thm:main}; it contains every vector with $c_d=0$, and both estimates \eqref{eq:main-fixed} and \eqref{eq:main-fixed-star} hold for every $\cc\notin\mathcal E_d(H)$.
\end{proof}

\bibliographystyle{amsplain}
\enlargethispage{4\baselineskip}
\bibliography{Ref_revised}

\end{document}